\documentclass[11 pt, reqno]{article}

\usepackage{amsmath,amsfonts,amssymb,amsthm}
\usepackage[colorlinks=true,hyperindex=true]{hyperref}
\usepackage{cancel,bbm}
\usepackage{mathabx} 
\usepackage{comment}
\usepackage{enumitem}
\usepackage{cite}

\usepackage{chngcntr}
\counterwithin*{equation}{section}

\usepackage[left=3cm, right=3cm, top=3 cm, bottom= 3 cm]{geometry}
\usepackage{color}
\usepackage{fancyhdr}
\usepackage{latexsym}

\newtheorem{ccounter}{ccounter}[section]
\newtheorem{theorem}[ccounter]{Theorem}
\newtheorem{lemma}[ccounter]{Lemma}
\newtheorem{cor}[ccounter]{Corollary}
\newtheorem{definition}[ccounter]{Definition}
\newtheorem{prop}[ccounter]{Proposition}
\newtheorem{ass}[ccounter]{Assumption}
\newtheorem{ex}[ccounter]{Example}

\def\bet{\begin{theorem}}
\def\eet{\end{theorem}}
\def\bel{\begin{lemma}}
\def\eel{\end{lemma}}
\def\bas{\begin{ass}}
\def\eas{\end{ass}}
\def\bec{\begin{cor}}
\def\eec{\end{cor}}
\def\bed{\begin{definition}}
\def\eed{\end{definition}}
\def\bep{\begin{prop}}
\def\eep{\end{prop}}
\def\beq{\begin{equation}}
\def\eeq{\end{equation}}
\def\proof{\noindent {\bf Proof.}\ \ }
\def\bea{\begin{equation*}}
\def\eea{\end{equation*}}

\def\bex{\begin{ex}}
\def\eex{\end{ex}}
\def\remark{\noindent{\bf Remark. }}

\def\rr{\mathbb{R}}

\def\1{\boldsymbol{1}}

\def\e{\mathrm{e}}

\def\del{\partial}
\def\d{\mathrm{d}}
\def\eps{\varepsilon}
\renewcommand\leq\varleq
\renewcommand\geq\vargeq
\def\ee{\mathrm{E}}

\def\F{\mathcal{F}}
\def\O{\mathcal{O}}

\def\ee{\mathbb{E}}

\def\pp{\mathbb{P}}

\def\tilE{\tilde{E}}

\def\Var{\mathrm{Var}}
\def\Cov{\mathrm{Cov}}

\newcommand{\psiV}{\psi^{V}}
\def\Deld{\Delta^{(d)}}
\def\Deli{\Delta^{(i)}}
\def\Deldd{\Delta^{(dd)}}
\def\Delii{\Delta^{(ii)}}
\def\Deldi{\Delta^{(di)}}
\def\Delid{\Delta^{(id)}}
\def\udd{u^{(\theta \theta )}}
\def\udi{u^{ ( \theta \eta ) }}
\def\uii{u^{(\eta \eta ) } }
\def\Eet{E^{(\eta , \theta)}}
\def\hf{h^{(f)}}
\def\tilE{\tilde{E}}
\def\hatE{\hat{E}}
\def\Ett{E^{(\theta, \theta )}}
\def\tilW{\tilde{W}}
\def\tilu{\tilde{u}}
\def\tilh{\tilde{h}}
\def\tilk{\tilde{k}}

\begin{document}
\title{Deformed GOE}

\begin{table}
\centering

\begin{tabular}{c}

\multicolumn{1}{c}{\parbox{12cm}{\begin{center}\Large{\bf KPZ-type fluctuation exponents for interacting diffusions in equilibrium }\end{center}}}\\
\\
\end{tabular}
\begin{tabular}{ c c c  }

 Benjamin Landon & Christian Noack &  Philippe Sosoe
 \\
 & & \\  
  \footnotesize{University of Toronto} & \footnotesize{Purdue University} & \footnotesize{Cornell University}  \\
  \footnotesize{Department of Mathematics}  & \footnotesize{Department of Mathematics} & \footnotesize{Department of Mathematics}  \\
 \footnotesize{\texttt{blandon@math.toronto.edu}} & \footnotesize{\texttt{cnoack@purdue.edu}} &\footnotesize{\texttt{psosoe@math.cornell.edu}} \\
  & & \\
\end{tabular}
\\
\begin{tabular}{c}
\multicolumn{1}{c}{\today}\\
\\
\end{tabular}

\begin{tabular}{p{15 cm}}
\small{{\bf Abstract:} We consider systems of $N$ diffusions in equilibrium interacting through a potential $V$. We study a ``height function'' which for the special choice $V(x) = \e^{-x}$, coincides with the partition function of a stationary semidiscrete polymer, also known as the (stationary) O'Connell-Yor polymer.  For a general class of smooth convex potentials (generalizing the O'Connell-Yor case), we obtain the order of fluctuations of the height function by proving matching upper and lower bounds for the variance of order $N^{2/3}$, the expected scaling for models lying in the KPZ universality class.  The models we study are not expected to be integrable and our methods are analytic and non-perturbative, making no use of explicit formulas or any results for the O'Connell-Yor polymer.}
\end{tabular}
\end{table}

\section{Introduction}

The\let\thefootnote\relax\footnote{The work of B.L. is partially supported by an NSERC Discovery grant. The work of P.S. is partially supported by NSF grants DMS-1811093 and DMS-2154090.} O'Connell-Yor polymer \cite{OY}, also known as the semi-discrete polymer, is a central model of the Kardar-Parisi-Zhang (KPZ) universality class. The model is an ensemble of up-right paths in a random environment formed by independent standard Brownian motions $B_1,\ldots, B_N$, with partition function
\begin{equation}\label{eqn: polymer-partition}
  Z_{N,t}(\beta) :=\int_{0 \leq s_1\leq \ldots \leq s_{N-1}\leq t} \e^{\beta \sum_{j=1}^{N} B_j(s_j)-B_j(s_{j-1})} \,\mathrm{d}s_1\dots \mathrm{d}s_{N-1},
\end{equation}
where we use the conventions $s_0=0$ and $s_N=t$. The free energy is defined by $\log (Z_{N, t} (\beta))$.  This model, along with a stationary version also introduced in \cite{OY} (see Section \ref{sec: stationary} below for the definition), has been an object of intense study over the past decade. Using the stationary version, O'Connell and Moriarty \cite{OM} computed the limiting free energy density
\beq
\lim_{N \to \infty} \frac{1}{N}\log Z_{N,t}(\beta).
\eeq
Sepp\"al\"ainen and Valk\'o \cite{SV} showed that the fluctuations of the free energy are of order $N^{\frac{1}{3}}$ for both the stationary and non-stationary models when $N$ and $t$ are tuned in a certain \emph{characteristic direction} (otherwise the fluctuations are Gaussian and of larger order
).  O'Connell \cite{O-toda} introduced a multidimensional diffusion process related to Dyson Brownian motion such that the law of $\log Z_{N,t}$ is equal to that of the ``leading particle" of the process, and used this to give a contour integral expression for its distribution.
An alternate contour representation was used by Borodin, Corwin and Ferrari \cite{BCF} to show that the free energy 
asymptotically has  Tracy-Widom fluctuations, confirming the expectation that the model belongs to the KPZ universality class. Vir\'ag \cite{V} shows convergence of a suitably centered and rescaled version of $\log Z_{N,t}(1)$, as well as the KPZ equation, to the KPZ fixed point of Matetski, Remenik and Quastel \cite{MQR}.  In this context we also note the concurrent and independent work of Sarkar and Quastel \cite{SQ} obtaining convergence to the KPZ fixed point for a broad class of exclusion processes, as well as the KPZ equation itself. An alternative (but equivalent \cite{NQR2}), more geometric, description of the scaling limit is the directed landscape, obtained as the continuum limit of Brownian Last Passage percolation, corresponding to $\beta=\infty$,  introduced by Dauvergne, Ortmann and Vir\'ag \cite{DOV}.

It has been noted by several authors (see for example \cite{FSW-book, O-toda, SS, spohn}) that the sequence 
\[v_j(t)=\log Z_{j,t}\] for 
$j=1,\ldots, N$ satisfies a system of stochastic differential equations of the form:
\begin{equation}\label{eqn: non-stat-equations}
\begin{split}
\mathrm{d}v_j&= -V'(v_j-v_{j-1})\mathrm{d}t+\mathrm{d}B_j,\\
V(x) &:= \e^{-\beta x}.
\end{split}
\end{equation}
In this setting, the implication of the Burke property discovered by O'Connell and Yor for their polymer model \cite{OY}, is that the solution   $\{ v_j(t)\}_{1\leq j \leq N}$ has an invariant measure of product form. The ``zero temperature" case, corresponding to the limit $\beta\rightarrow \infty$ of the system \eqref{eqn: non-stat-equations}, has been studied by Sasamoto-Spohn as well as Ferrari, Spohn and Weiss \cite{FSW1, FSW2, FSW-book, SS}. In this formal limit, the system consists of Brownian motions reflected off each other. Ferrari, Spohn and Weiss's results imply that for various classes of initial data, the distribution of the system has explicit expressions in terms of contour integrals that can be analyzed to find limiting distributions given by the Airy process. More recently, Nica, Remenik and Quastel \cite{NQR} showed that the scaling limit of the time-dependent system at zero temperature is described by the KPZ fixed point.

Systems such as \eqref{eqn: non-stat-equations}, as well as the equilibrium version we study below are the totally asymmetric analog of the following classical system of interacting Brownian motions studied in, e.g.,   \cite{spohn-CMP} (see \cite{harris} as well for the case $\beta=\infty$):
\begin{equation}\label{eqn: symmeq}
\mathrm{d}v_j= (V'(v_{j+1}-v_j)-V'(v_j-v_{j-1}))\mathrm{d}t+\mathrm{d}B_j.
\end{equation}
For general convex $V$, this is sometimes known as the Ginzburg-Landau system \cite{zhu}. Chang and Yau  \cite{changyau} famously derived the fluctuations for such processes out of equilibrium. In  this symmetric case, the  fluctuations are of order $N^{\frac{1}{4}}$. 

Diehl, Gubinelli and Perkowski \cite{DGP} study the weakly asymmetric case, where \eqref{eqn: symmeq} is replaced by
\[\mathrm{d}v_j= (pV'(v_{j+1}-v_j)-qV'(v_j-v_{j-1}))\mathrm{d}t+\mathrm{d}B_j\]
with $p-q=1/N$ and show that after suitably rescaling, the field 
\[u_j:=v_j-v_{j-1},\]  converges to a solution of the stochastic Burgers equation. Their result holds for a class of convex potentials with $V'$ Lipschitz. The recent work \cite{jaramoreno} specifically addresses the O'Connell-Yor case and proves convergence to the stochastic Burgers equation in the intermediate disorder regime.

The purpose of this paper is to probe the universality of the  fluctuations of an equilibrium version of the system \eqref{eqn: non-stat-equations}. As described above,  the system is exactly solvable and known to lie in the KPZ universality class in the special case $V(x) = \e^{-x}$.
 However, the KPZ universality is  conjectured to hold for generic classes of potentials, beyond the exactly solvable cases. For example, Ferrari, Spohn and Weiss \cite[Chapter 1, p. 3]{FSW2} write that ``the exponential [...] can be
replaced by `any' function  [of $v_j-v_{j-1}$] except for the linear one, and the system is
still in the KPZ universality class.'' Indeed, in the special case of quadratic potentials (i.e., $V'(x)$ linear) the system can be solved exactly and we will show that the fluctuations are Gaussian of order $N^{1/4}$ as in the symmetric case \eqref{eqn: symmeq}. 

In the main results of this paper, we will introduce a general  class of convex potentials $V$ which includes the O'Connell-Yor case $V(x)=\e^{-x}$, but  are not expected to be explicitly solvable.  The class of potentials we consider includes the Laplace transform of any finite measure compactly supported in $(0, \infty)$.   To each system of interacting diffusions we associate a ``height function'' that coincides with the polymer partition function in the O'Connell-Yor case.  We prove the variance of this height function is of order $\O (N^{2/3})$.  In particular, we recover the upper bounds for the known exponents for the O'Connell-Yor model \cite{SV,MFSV} entirely through a dynamical approach without appealing to the polymer representation \eqref{eqn: polymer-partition} or any exact formulas whatsoever, as both are unavailable for the class of models we consider.  Secondly, under a curvature condition (satisfied in the O'Connell-Yor case, and a condition that we in general expect to be generic) we complement our upper bounds with a lower bound for the variance of the same order of magnitude.  

Our upper and lower bounds provide evidence for the conjecture that this model lies in the KPZ universality class by exhibiting the correct order of fluctuations.  After  introducing our model and results we will state perspectives on why obtaining the full universality for this model may indeed be tractable.

The strategy of proof is inspired by coupling arguments appearing in works of Bal\'azs-Cator-Sepp\"al\"ainen for the corner growth model with exponential weights \cite{BCS}, Bal\'azs-Sepp\-{\"a}l{\"a}inen for the ASEP \cite{BS}, Bal\'azs-Komj\'athy-Sepp\"al\"ainen on zero-range and deposition processes \cite{BKS1, BKS2}, Bal\'azs-Sepp\"al\"ainen-Quastel on the KPZ equation \cite{BSQ} and Sepp\"al\"ainen on various models of last passage percolation and polymers in random environments \cite{S1, S2}. All of these models are expected to belong to the KPZ universality class, and in some cases this expectation has been confirmed by rigorous results such as the existence of asymptotic random matrix (e.g. Tracy-Widom) fluctuations.  

The difference between the model we consider here and those mentioned thus far -- other than the discrete nature of the state space in most of these works -- is that, without the polymer or particle system interpretation, we do not have access to quantities which play the role of the occupation length or second class particles when we work with perturbations of the initial data.  We therefore must rely entirely on properties of the system that may be deduced from the dynamical interpretation \eqref{eqn: non-stat-equations} and the evolution equations we will later derive for the perturbations.

The models we consider depend on a parameter $\theta >0$, essentially controlling a drift in one of the Brownian driving terms in \eqref{eqn: non-stat-equations}. In \cite{CN1}, the last two authors  gave an alternative proof of the result of Sepp\"al\"ainen and Valk\'o that the variance of the O'Connell-Yor polymer is of $\O (N^{2/3})$. A substantial component of this argument is that, due to the partition function representation, the polymer is convex and monotone in $\theta$. The strategy here follows this proof in broad strokes; however, the required monotonicity and convexity are now highly non-trivial (lacking a polymer representation) and the argument now appears quite abstract. For reader convenience we recall the proof of \cite{CN1} in Section \ref{sec:upper-sketch}. 

Our proof of the lower bound relies on the introduction of a random functional which plays a similar role to the polymer Gibbs measure; for lack of a better term, we introduce the ``pseudo-Gibbs measure'' in Section \ref{sec:psg}. In particular, certain Malliavin derivatives of the height function  can be represented as expectations with respect to the pseudo-Gibbs measure, similar to the interpretation of a polymer partition function as a cumulant generating function. However, the connection seems to stop at the first derivative (second derivatives are not covariances) and so establishing various properties of the pseudo-Gibbs measure (e.g., monotonicity of expectations) takes place in a relatively abstract manner.

\section{Definition of model and statement of results}

Consider the following system of interacting diffusions
 \begin{equation}
 \label{eqn: system}
 \begin{cases}
    \mathrm{d}u_1 &= -V'(u_1)\mathrm{d}t+\mathrm{d}B_0 - \theta \d t +\mathrm{d}B_1\\
    \mathrm{d}u_j &= (V'(u_{j-1})-V'(u_j))\mathrm{d}t + \mathrm{d}B_j-\mathrm{d}B_{j-1}, \quad 2 \leq j\leq N
 \end{cases}
\end{equation}
Here $B_0,\ldots, B_N$ are independent standard Brownian motions on $\mathbb{R}_+$ and $ \theta >0$. 
 We are interested in the case where the initial data \[u(0)=\big(u_1(0),\ldots,u_N(0)\big)\]
is distributed according to the unique invariant measure for \eqref{eqn: system}, which is a product measure of the form,
\beq \label{eqn:inv-meas-def}
\omega_\theta (x_1, \dots, x_N) := \frac{1}{ Z(\theta)^N} \prod_{j=1}^N \e^{ - \theta x_j - V(x_j ) } =: \prod_{j=1}^N \nu_\theta (x_j) .
\eeq
Here, $Z(\theta)$ is the normalization constant,
\beq \label{eqn: Z-def}
Z( \theta) := \int_{\rr} \e^{ - \theta x - V(x)} \d x
\eeq
and $\nu_\theta$ is the probability measure on $\rr$ defined implicitly above. The invariance can be easily seen at the level of formal calcuation by applying the adjoint of the generator of this diffusion to the above measure, and will be rigorously justified below (see Proposition \ref{prop:inv-measure}). 
The class of potentials $V$ we consider is as follows.


\bed \label{def: V-def} 
We say $V$ is of O'Connell-Yor-type if $V \geq 0$ is a smooth convex function satisfying,
\beq \label{eqn: assumption}
V(x) \geq c |x|^2 \1_{ \{ x \leq - C \} } , \qquad V'(x) \leq 0
\eeq
and 
\beq \label{eqn:deriv-assump}
c_0 V''(x) \leq - V''' (x) \leq \frac{1}{c_0} V'' (x) + C \1_{ \{ x \geq -C \} }
\eeq
\eed
\remark As a consequence of the assumptions \eqref{eqn: assumption} we have for any $\theta >0$ that,
\begin{equation} \label{eqn:Vthetalb}
V(x)+\theta x\geq c'|x| - C'
\end{equation}
for some positive $c', C'>0$.   \qed

\remark It is easy to see that if $\mu$ is a finite positive measure whose support is compactly contained in $(0, \infty)$, then
\beq
V_\mu (x) := \int \e^{ - x s} \d \mu (s)
\eeq
is of O'Connell-Yor type, as is  $V(x) := V_\mu (x) + \eps \varphi (x)$ for $\varphi \in C_c^\infty ( \rr)$ and $\eps$ sufficiently small. \qed

Our assumptions in fact imply an exponential growth condition on $V(x)$ and so existence of solutions to \eqref{eqn: system} does not lie within the standard theory. Nonetheless, this system is well-behaved and for completeness we give a proof of the following in Appendix \ref{a:diff-1}. One could prove the following under considerably less stringent conditions on $V$, but this would take us too far astride of the main goal of our work.

\bep \label{prop:inv-measure}
Let $V$ be of O'Connell-Yor type and $\theta >0$. The system \eqref{eqn: system} admits a unique, global-in-time strong solution that is a Markov process with unique invariant measure given by $\omega_\theta$ defined in \eqref{eqn:inv-meas-def} above. 
\eep

\subsection{Link with the O'Connell-Yor polymer}\label{sec: stationary}
Let us now explain the connection between the model introduced in \cite{OY} and the system \eqref{eqn: system} in equilibrium. The \emph{stationary semi-discrete polymer} is a polymer model in a random environment, defined by a variant of the partition function \eqref{eqn: polymer-partition}. Consider a collection $B_0,B_1,\ldots,B_N$ of $N+1$ \emph{two-sided} Brownian motions with $B_0(0) = 0$ and a parameter $\theta>0$. We then define
\begin{equation}\label{eqn: Z}
Z^\theta_{N,t} :=\int_{-\infty< s_0\leq  s_1\leq \ldots \leq s_{N-1}\leq t} \e^{\theta s_0 -B_0(s_0)+\sum_{j=1}^N B_j(s_j)-B_j(s_{j-1})} \,\mathrm{d}s_0\cdots \mathrm{d}s_{N-1},
\end{equation}
where $s_N=t$ as in \eqref{eqn: polymer-partition}, but $s_0$ is now a variable of integration. Imamura-Sasamoto \cite{sasamoto} derived a contour integral representation for the above model and found the Baik-Rains distribution for the limiting free energy distribution. 
A simple computation using It\^o's formula shows that the quantities \begin{align*}
    u_1^{\mathrm{OY}}(t)&:= \log Z^\theta_{1,t} +B_0(t)-\theta t,\\
    u_j^{\mathrm{OY}}(t)&:=\log Z^\theta_{j,t}-\log Z^\theta_{j-1,t}, \quad 2\leq j\leq N
\end{align*}
satisfy the stochastic differential equations \eqref{eqn: system} with 
\[V(x)=\e^{-x}.\] 
In particular, we note the following relation involving the free energy:
\begin{equation}\label{eqn: polymer-height}
\log Z_{N,t}^\theta= \sum_{j=1}^N u^{\mathrm{OY}}_j(t) -B_0(t)+\theta t.
\end{equation}
The main result of \cite{OY}, the \emph{Burke property}, implies that $u_j^{\mathrm{OY}}$ has a product form invariant measure. That is, if $u^{\mathrm{OY}}(0)=(u^{\mathrm{OY}}_1(0),\ldots,u^{\mathrm{OY}}_N(0))$ is an iid vector with the distribution $(\log \frac{1}{X_j})_{1\leq j\leq N}$, where $X_j$ is a $Gamma(\theta)$\footnote{These random variables have density $1_{\{x>0\}}x^{\theta-1} \e^{-x}\frac{\mathrm{d}x}{\Gamma(\theta)}$.} random variable, then $u^{\mathrm{OY}}(t)$ has the same distribution at later times. In particular,
\[\log Z_{N,t=0}^\theta=\sum_{j=1}^N u_j^{\mathrm{OY}}(0),\]
has the distribution of an iid sum.

\subsection{Observable and statement of results}
The main object of study in this paper is the analogue of the free energy in \eqref{eqn: polymer-height}.  Namely, let $V$ be an O'Connell-Yor type potential, and let $u_j (t)$ denote the solution to \eqref{eqn: system} with initial data distributed according to the product invariant measure $\omega_\theta$. Define,
\begin{equation}\label{eqn: height-def}
    W_{N,t}^\theta:= \sum_{j=1}^N u_j(t) -B_0(t)+\theta t.
\end{equation}
Recall $Z(\theta)$ as defined in \eqref{eqn: Z-def} and for $k \geq -1$, set
\begin{equation}\label{eqn: psi-def}
    \psi_k^V(\theta):=\frac{\mathrm{d}^{k+1}}{\mathrm{d}\theta^{k+1}}\log Z(\theta).
\end{equation}
Note that $\psiV_1(\theta)>0$, being the variance of a random variable distributed according to $\nu_\theta$. 

From \eqref{eqn: height-def} and the invariance of $\omega_\theta$, we have 
\begin{equation}\label{eqn: rough}
\mathrm{Var}(W_{N,t}^\theta)\leq 2\mathrm{Var}\big(\sum_{j=1}^N u_j(0)\big)+2t= \O(N+t).
\end{equation}
if the initial data is distributed according to $\omega_\theta$. In general, the order of this bound cannot be improved: by Corollary \ref{cor:bad-lower} below, we have
\beq
\mathrm{Var}(W_{N,t}^\theta)\geq |N\psiV_1(\theta)-t| ,
\eeq
so \eqref{eqn: rough} is of the correct order if either one of the two parameters $N$ or $t$ is much larger than the other. However, for special values of $N$ and $t$ depending on $\theta$, there is cancellation between the iid sum and the  Brownian motion term in \eqref{eqn: height-def}. For example, we will see below (see Proposition \ref{prop: normal}) that in the special case $V(x)=\frac{x^2}{2}$, and $t=N$, the variance can be computed exactly and the fluctuations are of order $N^{\frac{1}{4}}$. 

The main result of Sepp\"al\"ainen and Valk\'o \cite{SV} for the O'Connell-Yor polymer implies that if $t$ and $N$ are suitably chosen (see \eqref{eqn:  characteristic-direction}), the fluctuations are of order $N^{\frac{1}{3}}$, a growth rate characteristic of the Kardar-Parisi-Zhang universality class. Our first main result is a non-perturbative argument which extends the variance upper bound in \cite{SV} to a large class of potentials. Sepp\"al\"ainen and Valk\'o's proof relies on the polymer interpretation explained in Section \ref{sec: stationary}, which is not available for potentials other than $\e^{-\beta x}$, $\beta>0$.

Our upper bound is the following and is proven in Section \ref{sec:upper}. 
\begin{theorem}\label{thm: upper-bound}
   Let $V$ be a O'Connell-Yor type potential. 
    Fix $\theta >0$ and suppose  
     that $N$ and $t$ are chosen so that
    \begin{equation}\label{eqn:  characteristic-direction}
|t-N\psi_1^V(\theta)|\leq A N^{\frac{2}{3}},
\end{equation}
for some $A>0$.  
    Then, there exists $C>0$ such that
    \begin{equation}\label{eqn: variance-upper-bound}
        \mathrm{Var}^\theta(W^\theta_{N,t})\leq CN^{\frac{2}{3}}.
    \end{equation}
\end{theorem}
\remark In general, if one fixes $\theta$ or allows it to vary over a compact interval supported in $(0, \infty)$ one obtains the estimate,
\beq
\Var ( W_{N, t}^\theta ) \leq CN^{2/3} + | t- N \psiV_1 ( \theta ) |,
\eeq
for some $C>0$ and all $t >0$. \qed


We remark here that the assumptions \eqref{eqn: assumption} are primarily used to show that the system \eqref{eqn: system} is well-posed and that the solutions are differentiable with respect to various parameters.  We expect that our argument can likely be extended to any convex $V$ satisfying \eqref{eqn:deriv-assump} for which well-posedness and certain differentiability properties hold.  


If, instead of the characteristic direction condition \eqref{eqn:  characteristic-direction}, we assume that $|t-N\psiV_1(\theta)| $ is much larger than $N^{2/3}$, then the fluctuations are of larger order and are in fact Gaussian. The following is proven in Section \ref{sec:main-cor}.
\begin{cor}\label{main-cor}
Fix $\theta >0$ and suppose that $t = t_N$ is such that
\beq
\lim_{N \to \infty} \frac{ | t - \psiV_1 ( \theta ) | }{N^{2/3} } = \infty.
\eeq
Then,
\beq
\frac{ W_{N, t}^\theta - \ee[ W_{N, t}^\theta ]}{| t- \psiV_1 ( \theta ) |^{1/2}}
\eeq
converges to a standard Gaussian random variable as $N \to \infty$. 
\end{cor}

In the case where $V(x)$ is quadratic, the system \eqref{eqn: system} is linear and admits an explicit solution as a Gaussian process.  In this case, we have the following, proven in Section \ref{sec:gauss}. 
\begin{prop}[The Gaussian case]\label{prop: normal}
Let $V(x)=\frac{x^2}{2}$.  Then, for each $\theta>0$, the random variable $W_{N, t}^\theta$ is a Gaussian with variance,
\beq
\Var ( W_{N, t}^\theta ) = (N-t) \left( 1- 2 \int_0^t \frac{s^{N-1}}{(N-1)!} \e^{ -s } \d s \right) + 2 \frac{ t^N}{(N-1)!} \e^{ -t }.
\eeq
When $t=N$, we obtain that the normalized quantity
\begin{equation*}
\big(\frac{\pi}{2}\big)^{1/4}\frac{W_{N,N}^\theta}{N^{1/4}}
\end{equation*}
is asymptotically a standard normal random variable as $N \to \infty$.
\end{prop}

  We complement the upper bound in Theorem \ref{thm: upper-bound} with the following. The proof appears at the end of Section \ref{sec:lower-proof}. 
\bet \label{thm:lower}
 Assume that $V$ is of O'Connell-Yor type and let $\theta_0 >0$ satisfy $\psiV_2 ( \theta_0 ) < 0$. Then, there is a $c>0$ so that
 \beq
 \Var\left( W_{N, t}^{ \theta_0} \right) \geq \max\{ | t- N \psiV_1 ( \theta_0 ) | , c N^{2/3} \}
 \eeq
\eet
The derivative $\psiV_2 ( \theta)$ equals (minus) the third central moment of a random variable distributed according to $\nu_\theta$. Lemma \ref{lem:psiV2} proves via a dynamical argument that $\psiV_2 ( \theta ) \leq 0$ for O'Connell-Yor type potentials. In the case $V(x) = \e^{-x}$ in fact $ (-1)^k \psiV_k ( \theta) >0$ for all $k \geq 2$. In the Gaussian case $V(x) = \frac{x^2}{2}$ we have $\psiV_2 ( \theta ) =0$ and the fluctuations are of lower order.

\section{Tools and techniques}
 In \cite{CN1}, the last two authors  gave an alternative proof of the result of Sepp\"al\"ainen and Valk\'o (the $N^{1/3}$ fluctuations of the stationary O'Connell-Yor polymer). This relied on interpreting the partition function as a log-moment generating function, allowing for the easy proof of certain monotonicity and convexity properties of the underlying couplings. For example, log moment generating functions are typically trivially seen to be convex functions, their second derivatives being a variance of the underlying random variable.

    The argument we give proving our  upper bound will follow in broad strokes that given in \cite{CN1}. For this to work, we will require  perturbations of the system under various parameters to have certain monotonicity and convexity properties.  This requires careful choice of the underlying couplings. In this section we will introduce our couplings and then give an outline of the proof of our upper bound. We will then discuss other aspects of our proofs. 

\subsection{Couplings}

For notational simplicity let us choose a realization of the Brownian motions driving \eqref{eqn: system} so that they are continuous for every point in the underlying probability space; for example the canonical realization on Wiener space suffices.  As a result of Appendix \ref{a:diff-1} we have the following.
\bep
For every choice of the initial data and every realization of the Brownian motions the system \eqref{eqn: system} has a unique strong solution $\{u_j \}_{j \geq 1}$ in that they are continuous functions satisfying,
\begin{align}
u_1 (t) - u_1(0) &= - \int_0^t (V'(u_1 (s) ) + \theta ) \d s + B_0 (t) + B_1 (t) \notag \\
u_j (t) - u_j (0) &= \int_0^t ( V' (u_{j-1} (s) ) - V' (u_j (s) ) ) \d s + B_j (t) - B_{j-1} (t) , \qquad j \geq 2
\end{align}
 for every $t \geq 0$.  
\eep
Given the previous, we can introduce the following couplings. 
\bed \label{def:coupling}
Let $\{q_i \}_{i=1}^\infty$ be iid uniform $(0, 1)$ random variables. For any $\theta >0$ define,
\beq
F_\theta (x) := \frac{1}{ Z ( \theta ) } \int_{-\infty}^x \e^{ - \theta u - V (u) } \d u,
\eeq
For any $\eta, \theta >0$ we now define $u_j (t, \eta, \theta)$ to be the solution of \eqref{eqn: system} with initial data,
\beq
u_j (0, \eta, \theta ) := F_\theta^{-1} ( q_j ).
\eeq
Note that $\{ u_j (0, \eta, \theta)\}_{j=1}^N$ are distributed as the invariant measure $\omega_\eta$ of the system \eqref{eqn: system} with $\theta = \eta$.
\eed
We will refer to $\eta$ as the \emph{initial data} parameter and $\theta$ as the \emph{driving} parameter. Given $\eta, \theta >0$ we introduce the two-parameter height function by,
 \beq \label{eqn:W-def}
 W_{N, t} ( \eta, \theta ) := \sum_{j=1}^N u_j (t, \eta, \theta ) - B_0 (t) + \theta t,
 \eeq
 so that the height function $W^\theta_{N, t}$ defined in \eqref{eqn: height-def} has the distribution $W_{N, t}^\theta \stackrel{d}{=} W_{N, t} ( \theta , \theta)$. 

\subsection{Variance formula}
In Section \ref{sec:representations}, we derive two representations for the variance in terms of derivatives of the height function with respect to the initial data and the parameter in the equations, respectively. For example, we have
\begin{equation}\label{eqn: rep}
\mathrm{Var}(W_{N,t}^\theta)=N\psiV_1(\theta)-t+2\mathbb{E} [\partial_\theta W_{N,t}(\eta,\theta)|_{\eta=\theta}].
\end{equation}
Note that $N \psiV_1 (\theta)-t = \O (N^{2/3})$ under the characteristic directions assumption \eqref{eqn:  characteristic-direction} on $N$ and $t$, so the main difficulty in obtaining the upper bound is to show that the term
\beq \label{eqn:show-this}
\mathbb{E} [\partial_\theta W_{N,t}(\eta,\theta)|_{\eta=\theta}]
\eeq
is of order $N^{\frac{2}{3}}$.


\subsection{Upper bound} \label{sec:upper-sketch}
In Section \ref{sec:upper}, we derive the upper bound. As stated, we must show that the term \eqref{eqn:show-this} is  $\O ( N^{2/3} )$.  At first glance, the computations we carry out in Section \ref{sec:upper} may be puzzling. 
In order to explain our proof, we will discuss in this section the O'Connell-Yor case $V(x) = \e^{-x}$.   In this case, the last two authors previously gave an alternative proof of the upper bound in \cite{CN1}, which served as the initial inspiration for our computations. 

For the O'Connell-Yor polymer, $W_{N,t}(\theta,\eta)$ has the same distribution as $\log Z_{N,t}^{\theta,\eta}$, where
\begin{equation}\label{eqn: 2-parameter-Z}
Z_{N,t}^{\theta,\eta} :=\int_{-\infty< s_0\leq  s_1\leq \ldots \leq s_{N-1}\leq t} \e^{\theta s_0^+-\eta s_0^- -B_0(s_0)+\sum_{j=1}^N B_j(s_j)-B_j(s_{j-1})} \,\mathrm{d}s_0\cdots \mathrm{d}s_{N-1}.
\end{equation}
Here $s_0^+=\max\{0,s_0\}$ and $s_0^-=\max\{0,-s_0\}$ are the positive and negative parts of $s_0$. Compare this to the log partition function \eqref{eqn: Z}, which equals $\log Z_{N,t}^{\theta,\theta}$. In this case,
\begin{equation}\label{eqn: partial}
\mathbb{E}[\partial_\theta W_{N,t}(\eta,\theta)|_{\eta=\theta}]=\mathbb{E}[E^\theta_{N,t}[s_0^+]] .
\end{equation}
Here $E_{N,t}^\theta$ denotes the random measure associated to the partition function \eqref{eqn: Z}:
\begin{align} \label{eqn:oy-gibbs}
    &E_{N,t}^\theta[f(s_0,\ldots,s_{N-1})] \notag\\
    :=&\frac{1}{Z_{N,t}^{\theta,\theta}}\int_{-\infty< s_0\leq  s_1\leq \ldots \leq s_{N-1}\leq t} \e^{\theta s_0 -B_0(s_0)+\sum_{j=1}^N B_j(s_j)-B_j(s_{j-1})} f(s_0,\cdots,s_{N-1})\,\mathrm{d}s_0\cdots \mathrm{d}s_{N-1}.
\end{align}
The key point is that the total derivative
\begin{equation}\label{eqn: total-der}
\frac{\mathrm{d}}{\mathrm{d}\theta} W_{N,t}(\theta,\theta)=E^\theta_{N,t}[s_0],\end{equation}
is easier to control than $\mathbb{E}[E_{N,t}^\theta[s_0^+]]$. For example, we have
\beq
\mathbb{E}[E^\theta_{N,t}[s_0]]=N\psi_1(\theta)-t=\O(N^{\frac{2}{3}}), 
\eeq
under the characteristic direction assumption \eqref{eqn:  characteristic-direction}. Here and below we let $\psi_k = \psiV_k$ in the O'Connell-Yor $V= \e^{-x}$ case. 
Writing
\begin{align}
(E_{N,t}^\theta[s_0])^2=~&(\partial_{\theta}^2 W_{N,t}(\theta,\eta)|_{\eta=\theta})^2+(\partial_{\eta}^2 W_{N,t}(\theta,\eta)|_{\eta=\theta})^2 \nonumber \\
&+2\,\partial_{\theta} W_{N,t}(\theta,\eta)|_{\eta=\theta}\cdot \partial_{\eta} W_{N,t}(\theta,\eta)|_{\eta=\theta} \label{eqn: cross}\\
=~&(E_{N,t}^\theta[s_0^+])^2+(E_{N,t}^\theta[s_0^-])^2 - 2E_{N,t}^\theta[s_0^+]E_{N,t}^\theta[s_0^-], \notag\\
\geq~&(E_{N,t}^\theta[s_0^+])^2 -2E_{N,t}^\theta[s_0^+]E_{N,t}^\theta[s_0^-]\label{eqn: macron}
\end{align}
we see that, to replace the $s_0^+$ by $s_0$, we must estimate the  cross term $E_{N,t}^\theta[s_0^+]E_{N,t}^\theta[s_0^-]$ from above.
This error term can be expected to be small. This is because it turns out that, with respect to the random measure $E_{N,t}^\theta$, $s_0$ is concentrated around $E_{N,t}^\theta[s_0]$ on a much smaller scale $\O(N^{1/2})$ than the typical size and standard deviation of $E_{N,t}^\theta[s_0]=\O(N^\frac{2}{3})$ with respect to $\mathbb{P}$. More precisely, by invariance, we have:
\begin{equation}\label{eqn: VV}
N \psi_2 ( \theta ) = \frac{\mathrm{d}^2}{\mathrm{d}\theta^2}\ee \left[ W_{N,t}(\theta,\theta) \right]= \ee \left[ E_{N,t}[(s_0-E_{N,t}[s_0])^2] \right],
\end{equation}
showing that the quenched variance is $\O (N)$ in expectation. 
This means that, for a given realization of the Brownian motions $B_0,\ldots,B_N$, the quantity $s_0$ lies within $N^{1/2}$ of $E_{N,t}^\theta[s_0]$ which is of the order $N^{2/3}$. Therefore, 
 $s_0$ is very likely to have the same sign as this expectation, so that $E_{N,t}^\theta[s_0^+]E_{N,t}^\theta[s_0^-]$ should be small. The observation in \cite{CN1} making this picture rigorous is that the RHS of \eqref{eqn: VV} equals
\begin{equation}\label{eqn: disjoint}
\mathbb{E}\big[E_{N,t}[(s_0^+-E_{N,t}[s_0^+])^2]\big]+\mathbb{E}\big[E_{N,t}[(s_0^--E_{N,t}[s_0^-])^2]\big]+2 \mathbb{E}\big[E_{N,t}[s_0^+] E_{N,t}[s_0^-]\big],
\end{equation}
due to the fact that $s_0^+$ and $s_0^-$ have disjoint support. This shows that the cross-term on the last line of \eqref{eqn: macron} is at most $\O(N)$, an acceptable error compared to the target bound for $\ee \left[ E_{N,t}^\theta[s_0^+] \right]^2=\O(N^{\frac{4}{3}})$.

Next, once one has replaced $\del_\theta W_{N, t}$ by the total derivative, we can apply  convexity of $\theta \mapsto W_{n, t} ( \theta, \theta)$. That is, the derivative can be bounded in terms of difference quotients: 
\beq \label{eqn: rhs-sketch}
\frac{\mathrm{d}}{\mathrm{d}\theta} W_{N,t}(\theta,\theta)\leq \frac{W_{N,t}(\theta,\theta)-W_{N,t}(\lambda,\lambda)}{\lambda-\theta} ,\eeq 
where $\lambda = \theta \pm N^{-1/3}$ (depending on if one is estimating the total derivative above or below in magnitude). So, this implies that $E^\theta_{N,t}[s_0^+]^2$ can be controlled by the square of the RHS of \eqref{eqn: rhs-sketch}. But these quantities can be related back to the variance of $W_{N, t}^\theta$ and $W_{N, t}^\lambda$. In order to move the $\lambda$ \emph{back} to $\theta$, we use the estimates,
\beq
\mathbb{E}[W_{N,t}(\theta,\theta)]-\mathbb{E}[W_{N,t}(\lambda,\lambda)]=\O(N^{\frac{1}{3}})
\eeq
which uses the characteristic direction assumption (see \eqref{eqn:expect-comp} in the general setting of this paper)
as well as a similar estimate for perturbing the variance (see Lemma \ref{lem:var-comp}). Hence, we obtain
\beq
\mathbb{E}[E^\theta_{N,t}[s_0^+]] \leq \mathbb{E}[E^\theta_{N,t}[s_0^+]^2]^{1/2} \leq C(N^{1/3} ) ( \mathcal{V} + N^{1/3} ) 
\eeq
 where $ \mathcal{V}=\mathrm{Var}^\theta(W^\theta_{N,t})$. Combined with the variance representation \eqref{eqn:  rep} we obtain,
\beq
\mathcal{V}\leq CN^{\frac{1}{3}}(\mathcal{V}^{\frac{1}{2}}+N^{\frac{1}{3}})
\eeq
which completes the proof. 

Reproducing the above strategy in our setting requires overcoming a number of serious obstacles. We list the most important ones below.
\begin{enumerate}[label=(\arabic*)]
    \item In the O'Connell-Yor case, the expression \eqref{eqn: 2-parameter-Z} provides a natural coupling between the solutions with different inital data that is monotonone and convex in $\eta$, crucial properties for the proof. No such coupling is provided \emph{a priori} for the system \eqref{eqn: system}.
    \item For general potential $V$, the height function $W_{N,t}(\theta,\theta)$ is only defined by analogy with the O'Connell-Yor log partition function. It is not clear what the meaning of the Gibbs measure $E^\theta_{N,t}[\cdot]$ or $s_0$ should be in the general case. The interpretation of the derivatives in terms of expectations is used crucially in several places like \eqref{eqn: macron}.
    \item Given a suitable coupling of the initial data, the derivatives of $W_{N,t}(\theta,\theta)$ in the case of general $V$ share certain good features, such as their sign and monotonicity, with the corresponding quantities in the O'Connell-Yor case. We exploit these more fully in our proof of the lower bound. However, there is no exact relation analogous to \eqref{eqn: disjoint} between the second derivative \eqref{eqn: VV} and the cross-term in \eqref{eqn: cross} in the general $V$ case, and so a different approach must be used to estimate this term.
\end{enumerate}

\subsection{Lower bound} 

Section \ref{sec:lower} contains our proof of Theorem \ref{thm:lower}, our lower bound.  We will see that, due to the lower bound $|N \psiV_2 ( \theta ) - t|$ for the variance, it suffices to consider the case that $t = N \psiV_2 ( \theta)$. Our proof draws some inspiration from \cite{MFSV}. This is essentially a change of measure argument which attempts to make the event $\{ W_{N, t}^\theta - \ee[ W_{N, t}^\theta] \geq c N^{1/3} \}$ typical by relating $W_{N, t}^\theta$ to $W_{N, t}^\lambda$, for $\lambda = \theta + N^{-1/3}$. 

The main observation we use from \cite{MFSV} is that, due to the Cameron-Martin theorem and Cauchy-Schwarz, one can change the driving parameter $\theta$ to $\lambda$ for $s \in [0,  T N^{2/3}]$ for any $T>0$ at the cost of an overall \emph{multiplicative} constant in the probabilities. After doing so, one still must change the driving parameter in the remaining range of $s \geq TN^{2/3}$ as well as in the initial data.  The initial data change turns out to be monotonic in the correct direction, so only the large time regime must be handled. 

In the O'Connell-Yor case, this perturbation of the parameter would be related to the quenched probability that $\{s_0 > T N^{2/3} \}$. One would require a bit more concentration than is contained in the estimate of $\O (N^{2/3})$ for the variance of $W_{N, t}^\theta$ (as the variance decomposition \eqref{eqn:  rep} provides an estimate for the annealed expectation of $s_0^+$). For example, one could compute a fourth central moment. 

In our case, our argument rests on our discovery of a random functional that we call the \emph{pseudo-Gibbs} measure, and denote by $\Eet_{N, t}$. This is introduced in Section \ref{sec:psg}. It has an explicit form as an integral over the $N$-simplex of positive times and is reminiscent of \eqref{eqn:oy-gibbs}. (One could extend it to negative times with a bit more work but this is unnecessary for us.) While in fact $\del_\theta W_{N, t} ( \eta, \theta) = \Eet_{N, t} [ s_0^+]$ as for the O'Connell-Yor polymer, this is no longer true for higher derivatives and in general deducing desired properties of $\Eet_{N, t}$ as function of various parameters requires indirect reasoning. Nonetheless, it defines a measure on the simplex with total mass less than $1$ and so one can apply standard analytic arguments, e.g., Cauchy-Schwarz, etc.

In order to control the perturbation of $\theta$ for large times, we prove moderate deviation exponential tail estimates of $s_0^+$ with respect to the annealed pseudo-Gibbs measure.  Here, we roughly follow an argument of Emrah-Janjigian-Seppalainen \cite{EJS}. A crucial input is that, in analogy with a result of  Rains
\cite{rains} and Emrah-Janjigian-Seppalainen \cite{EJS}  for Last Passage Percolation with exponential weights on $\mathbb{Z}^2$ (the discrete, ``zero temperature'' version of the O'Connell-Yor polymer), one can derive an exact expression for
\beq \label{eqn:int-ejs}
\ee\left[ \exp \left(  ( \eta - \theta ) W_{N, t} ( \eta , \theta ) \right) \right] 
\eeq
using the Cameron-Martin formula. 

We remark here that under the additional assumption $\psi_2^V(\theta)<0$, our upper bound \eqref{eqn: variance-upper-bound} for the variance also follows from the estimates we obtain for the pseudo-Gibbs measure. Since the derivation of these estimates is at least as involved as the proof of the upper bound (compare for example Propositions \ref{prop:mixed} and \ref{lem: pseudo-monotone}), and since the upper bound can be obtained without introducing additional definitions or assumptions,  we present separate proofs for the upper and lower bounds.

\subsection{Organization of the remainder of paper}

In Section \ref{sec:upper} we will give the proof of Theorem \ref{thm: upper-bound}, our upper bound for the variance of $W_{N, t}^\theta$, assuming a number of intermediate results which are stated in the course of the proof, and whose proofs are given in Section \ref{sec:aux-upper}. The proof given in Section \ref{sec:upper} follows along the general lines of the argument sketched above in the O'Connell-Yor case. 

The various auxilliary results proven in Section \ref{sec:aux-upper} are as follows. Our representation of the variance is stated as Lemma \ref{lem:var-rep} and is proven in Section \ref{sec:representations}. Various monotonicity properties of the first derivatives are stated and proven in Section \ref{sec:aux-1d}.  In Section \ref{sec:aux-2d} we collect the various properties of the second derivatives of $W_{N, t} ( \eta , \theta)$ that we need. In particular, our substitutes for the elementary convexity in the O'Connell-Yor case, which are Lemmas \ref{lem:2nd-der-basic} and \ref{lem:2nd-ii}, are proven in Section \ref{sec:aux-2nd-basic}. Our result that treats the analog of the cross term on the last line of \eqref{eqn: macron} is Proposition \ref{prop:mixed} and is proven in Section \ref{sec:aux-2nd-mixed}.  Finally, a result concerning the stability of the variance under change of parameters is stated as Lemma \ref{lem:var-comp} and is proven in Section \ref{sec:aux-var-comp}.

Our lower bound, Theorem \ref{thm:lower}, is proven over the course of Sections \ref{sec:psg} and \ref{sec:lower}. In Section \ref{sec:generating} we use the Cameron-Martin theorem to derive an exact expression for certain exponential moments of $W_{N, t} ( \eta , \theta)$, the quantity appearing above in \eqref{eqn:int-ejs}. The pseudo-Gibbs measure is introduced in Section \ref{sec:psg-2} and in Section \ref{sec:psg-conc} we derive an upper tail bound for the random variable $s_0$ under the pseudo-Gibbs measure. In order to prove our lower bound, we introduce a three-parameter height function in Section \ref{sec:three-param} which is only a slight generalization of $W_{N, t} ( \eta, \theta)$. The main argument in the lower bound is given in Proposition \ref{prop:lower}, as it deals with the case of vanishing characteristic direction, i.e., when the quantity $t - N \psiV_1 ( \theta ) =0$. The general case is an easy corollary and so the full proof of Theorem \ref{thm:lower} is given in Section \ref{sec:lower-proof}. 

In Section \ref{sec:gauss} we deal with the Gaussian case, giving the short proof of Proposition \ref{prop: normal}. This boils down to a calculation of the variance as an integral, which in the case $N=t$ may be analyzed via Stirling's formula.

For ease of presentation, we have deferred most of the ``soft'' analysis to the appendices. These include the proof of well-posedness of the equations \eqref{eqn: system} and the determination of the invariant measure, which are both handled in Appendix \ref{a:diff-1}. Since the ``drift" terms $V'(u_j)$ appearing in the equation can have super-linear growth, the existence of global solutions does not follow directly from the most basic existence theorems for SDEs in the case where $V$ is of O'Connell-Yor type. However, our assumptions on $V(x)$ imply that it is strongly confining for negative $x$ and sublinear for $x>0$. Two further facts which simplify the analysis are the constant diffusion coefficients and the triangular nature of the system, that is, for each $k$, the first $k$ equations form a closed system.

As can be seen from the discussions above, we will often need to differentiate various quantities with respect to the parameters $\eta$ and $\theta$.  The differentiability is treated in Appendix \ref{a:diff}. If the reader accepts differentiability, then this section can be safely ignored, except for Appendix \ref{a:ode} which contains a few elementary properties of triangular systems of ordinary differential equations. The punch-line is that the various derivatives of the $u_j$ and $W_{N, t}$ will satisfy triangular systems of ODEs, resulting in various a priori bounds.  Appendix \ref{a:1b} contains some elementary calculations using calculus primarily for the purposes of showing that derivatives of the initial data with respect to $\eta$ (recall the coupling given in Definition \ref{def:coupling} above) have finite moments. This is also required to differentiate under the integral sign in a few places throughout our proofs. Finally Appendix \ref{a:psiV2} contains the proof that $\psiV_2 ( \theta ) \leq 0$ under our assumptions.

\subsection{Future perspectives}

Exhibiting the KPZ universality for non-integrable random growth models remains a challenging research direction. In this context, the system of interacting diffusions studied in this work is particularly exciting as it may offer a tractable setting in which to prove such universality results. In particular, due to the simple way in which the randomness of the Brownian motion terms enter in the system \eqref{eqn: system}, it may be possible to apply some of the ideas that appear in recent works on ergodicity of Dyson Brownian motion \cite{LY1,LY2,LSY}. Here, the main point is that by coupling general systems to those in equilibrium (by allowing them to have the same underlying driving Brownian motion terms, similar to how we have coupled the systems at different parameters) one derives parabolic difference equations for their differences, allowing for the use of PDE methods to study the time to local equilibrium (e.g., the energy method).  The dynamical methods of random matrix theory (see \cite{EY}) have inspired recent studies of the universality of lozenge tilings of general domains \cite{A}, and it is still unclear what is the full range of applicability of these ideas.

Important to the dynamical approach is obtaining certain a priori bounds on the quantities in play, as these are used to estimate, e.g., coefficients in the derived parabolic equations. In the eigenvalue context these are known as the local laws or rigidity results. In our setting, the variance estimates may be viewed as progress towards such estimates. In \cite{CN1}, the last two authors obtained concentration estimates for the O'Connell-Yor polymer. The starting point, a recursive scheme for estimating higher moments based on lower ones via Gaussian-integration-by-parts, is applicable here (remarkably, such schemes - using cumulant expansions in the place of Gaussian-integration-by-parts - have also found much recent success in random matrix theory, and are used to derive rigidity results in this context see \cite{MDE}). Both the universality of our model as well as obtaining moderate deviations results (i.e., an exponential tail estimates similar to \cite{EJS}) are subjects of current investigation.

\section{Proof of upper bound} \label{sec:upper}

Recall the two-parameter height function $W_{N, t} ( \eta, \theta)$ as defined in \eqref{eqn:W-def}. We will denote derivatives with respect to the first parameter by $\del_\eta$ (the initial data parameter) and those with respect to the second parameter (the driving parameter) by $\del_\theta$. Note that we will evaluate the two-parameter height function on the diagonal $\theta =\eta$ and so the notations $(\del_\eta W_{N, t} ) ( \theta , \theta)$ and $(\del_\theta W_{N, t} ) ( \theta, \theta)$, etc., are understood.

Our starting point is the following representation for the variance.
\bel \label{lem:var-rep}
We have,
\begin{align} \label{eqn:var-rep}
\Var (W_{N, t} ( \theta, \theta ) ) &= N \psiV_1 ( \theta) - t + 2 \ee[ ( \del_\theta W_{N, t} ) ( \theta , \theta ) ] \notag\\
&= t - N \psiV_1 ( \theta)  - 2 \ee[ ( \del_\eta W_{N, t} ) ( \theta, \theta ) ].
\end{align}
\eel
The proof is given in Section \ref{sec:representations}.  From the first representation in \eqref{eqn:var-rep}, it suffices to prove the estimate,
\beq \label{eqn:main-pr-1}
\ee[ \del_\theta W_{N, t} ( \theta, \theta) ] \leq  CN^{2/3} + CN^{1/3} \left( \Var (W_{N, t} ( \theta, \theta ) ) \right)^{1/2}
\eeq
which will be the goal of the remainder of the proof. In order to handle the LHS, we apply Cauchy Schwarz and estimate,
\begin{align}
\ee[ ( \del_\theta W_{N, t} ) ( \theta, \theta )]^2 \leq~ &\ee[ ( \del_\theta W_{N, t} ) ( \theta, \theta )^2 ]  \notag\\
\leq ~& \ee \left[ \left( \frac{ \d }{ \d \theta} W_{N, t} ( \theta, \theta)\right) ^2 \right]  -2 \ee[ \del_\eta W \del_\theta W].
\end{align}
The cross term is difficult to estimate. We will prove the following two results which relate it to the second derivative.
\bep \label{prop:mixed}
For all $N, t$ and $\theta$ we have,
\beq
0 \leq -\left( \del_\eta W_{N,t}  \right)(\theta, \theta)  \times \left( \del_\theta W_{N, t} \right) (\theta, \theta)  \leq \left( \del_\eta \del_\theta W_{N, t}  \right)  (\theta, \theta) .
\eeq
\eep
\bel \label{lem:2nd-der-basic}
We have for all $\eta, \theta >0$ and $N$ and $t \geq 0$,
\beq \label{eqn:2nd-der-basic-1}
\left( \del_\theta^2 W_{N, t} \right) ( \eta, \theta) \geq 0, \qquad \left(  \del_{\eta} \del_{\theta} W_{N,t } \right) ( \eta, \theta ) \geq 0 .
\eeq
Uniformly in $\eta, \theta$ varying in compact subsets of $(0, \infty)^2$ we have for all $N, t\geq0$ that,
\beq
\ee\left[ \left( \del_\eta^2 W_{N, t}  \right) ( \eta , \theta) \right] \geq - C N.
\eeq
\eel
Proposition \ref{prop:mixed} is proven in Section \ref{sec:aux-2nd-mixed} and Lemma \ref{lem:2nd-der-basic} is proven in Section \ref{sec:aux-2nd-basic}.  Using the above results and the identity
\[2 \del_{\eta} \del_{\theta} W_{N,t } =  \frac{ \d^2 }{ \d \theta^2 } W_{N, t} ( \theta , \theta )-\del_\eta^2 W_{N, t} -\del_\theta^2 W_{N, t},\]
we have
\begin{align} \label{eqn:upper-c2}
\ee[ ( \del_\theta W_{N, t} ) ( \theta, \theta )^2 ]  &\leq C N + \ee[ \left( \frac{ \d }{ \d \theta} W_{N, t}  ( \theta, \theta)\right)^2 ]  + C \frac{ \d^2 }{ \d \theta^2 } \ee[ W_{N, t} ( \theta , \theta ) ] \notag\\
&\leq  CN +\ee[ \left( \frac{ \d }{ \d \theta} W_{N, t}  ( \theta, \theta)\right)^2 ]
\end{align}
where we used $\ee[ W_{N, t} ( \theta, \theta) ] = N \psiV_1 ( \theta)$ in the second line (interchange of derivative and expectation is justified in Proposition \ref{prop:dif-under}). Let now $\eps_0 = \theta/2$ and consider $\lambda$ s.t. $| \lambda- \theta| \leq \eps_0$. Then by the nonnegativity in \eqref{eqn:2nd-der-basic-1} we have by Taylor expansion,
\begin{align} \label{eqn:taylor-exp}
W_{N, t} ( \lambda, \lambda ) - W_{N, t} ( \theta, \theta) &\geq  ( \lambda - \theta) \frac{ \d }{ \d \theta} W_{N, t} ( \theta, \theta) \notag\\
&- ( \lambda - \theta)^2 \left( \sup_{ \theta' : |\theta- \theta'| \leq \eps_0 }  \left( (\del_\eta^2 W_{N, t} ) ( \theta', \theta' )  \right)_- \right) ,
\end{align}
where $(a)_- := 0 \vee (-a)$ denotes the negative part of $a$. 
For the quantity on the RHS we prove the following in Section \ref{sec:aux-2nd-basic}. 
\bel \label{lem:2nd-ii}
We have,
\beq
\ee\left[ \left( \sup_{ |\theta' - \theta| \leq \eps_0 } \left( (\del_\eta^2 W_{N, t} ) ( \theta', \theta' )  \right)_- \right)^2 \right] \leq C N^2
\eeq
\eel
Taking $\lambda_\pm := \theta \pm N^{-1/3}$ we see from \eqref{eqn:taylor-exp} (after moving the $(\lambda_\pm -\theta)$ to the LHS)
\begin{align} \label{eqn:upper-c1}
\ee\left[ \left( \frac{ \d }{ \d \theta} W_{N, t}  ( \theta, \theta)\right)^2 \right] &\leq C \big\{ N^{4/3} + N^{2/3} \ee[ ( W ( \lambda_+ , \lambda_+ ) - W_{N, t} ( \theta, \theta ) )^2] \notag\\
&+ N^{2/3} \ee[ ( W ( \lambda_- , \lambda_- ) - W_{N, t} ( \theta, \theta ) )^2] \big\}
\end{align}
Since $\del_\theta \ee[ W_{N, t} ( \theta, \theta) ] = N \psiV_1 (\theta) - t = \O (N^{2/3})$ by the assumption \eqref{eqn:  characteristic-direction} that we are in a characteristic direction, we see that
\beq \label{eqn:expect-comp}
\left| \ee[ W_{N, t} ( \lambda_\pm , \lambda_\pm )] - \ee[ W_{N, t} ( \theta, \theta ) ]  \right| \leq C N^{1/3}.
\eeq
In addition, we have the following estimate for comparing the variance of $W_{N, t} ( \lambda, \lambda)$ back to that of $W_{N, t} ( \theta, \theta)$. It is proven in Section \ref{sec:aux-var-comp}. 
\bel \label{lem:var-comp}
Uniformly for $\theta$ and $\lambda$ varying over compact subsets of $(0, \infty)^2$ we have,
\beq
|\Var ( W_{N, t} ( \theta, \theta ) ) - \Var (W_{N, t} ( \lambda, \lambda ) ) | \leq CN | \theta - \lambda|.
\eeq
\eel
Using Lemma \ref{lem:var-comp} as well as \eqref{eqn:expect-comp} we obtain,
\begin{align}
\ee[ ( W ( \lambda_\pm , \lambda_\pm ) - W_{N, t} ( \theta, \theta ) )^2] &\leq  2  \Var ( W_{N, t} ( \lambda_\pm , \lambda_\pm ) ) + 2 \Var ( W_{N, t} ( \theta , \theta ) )  \notag\\
&+4 \left( \ee[ W_{N, t} ( \lambda_\pm, \lambda_\pm ) ] - \ee[ W_{N, t} ( \theta, \theta ) ] \right)^2 \notag\\
& \leq C \Var( W_{N, t} ( \theta, \theta ) ) + C N^{2/3}
\end{align} 
Plugging this estimate into the RHS of \eqref{eqn:upper-c1}, which in turn is used to estimate the last term on the 
RHS of \eqref{eqn:upper-c2}, we see that we have derived,
\beq
\ee[ ( \del_\theta W_{N, t} ) ( \theta, \theta ) ]^2 \leq C N^{4/3} +C N^{2/3} \Var ( W_{N, t} ( \theta, \theta) )
\eeq
which is equivalent to \eqref{eqn:main-pr-1}. We have therefore derived the inequality,
\beq
 \Var ( W_{N, t} ( \theta, \theta) ) \leq C\left( N^{2/3} + N^{1/3} \left(  \Var ( W_{N, t} ( \theta, \theta) ) \right)^{1/2} \right)
\eeq
which, after applying Cauchy-Schwarz to the RHS, proves Theorem \ref{thm: upper-bound}. \qed

\section{Proofs of auxilliary results used in upper bound} \label{sec:aux-upper}

In this section we collect the proofs of the various estimates used in the proof of Theorem \ref{thm: upper-bound} given in Section \ref{sec:upper}. 

\subsection{Variance representation} \label{sec:representations}

In this section we prove our variance representation, Lemma \ref{lem:var-rep}. We comment here that existence of the derivatives is justified by Corollaries \ref{cor:1st-der-1} and \ref{cor:1st-der-2}. We first prove the following.
\bel \label{lem:ibp}
We have,
\beq
\ee[ B_0 (t) W_{N, t} ( \eta, \theta) ] = - \ee \left[ \left( \del_\theta W_{N, t}  \right)( \eta , \theta ) \right]
\eeq
\eel
\proof Let us temporarily indicate the explicit dependence of $W_{N, t} ( \eta , \theta)$ on the Brownian motion $B_0$ by introducing the notation $\tilde{W}_{N, t} ( \eta, \theta, B_0 )$. Then,
\beq
W_{N, t} ( \eta, \theta+h) = \tilde{W}_{N, t} ( \eta , \theta+h, B_0 ) = \tilde{W}_{N, t} (\eta, \theta , \tilde{B}_{0, h} )
\eeq
where,
\beq
\tilde{B}_{0, h} (s ) := B_0(s) - h s.
\eeq
By the Cameron-Martin formula,
\beq
\ee[ \tilde{W}_{N, t} ( \eta, \theta , \tilde{B}_{0, h} ) ] = \ee[ \tilde{W}_{N, t} ( \eta , \theta, B_0 ) \e^{ - h (B_0 (t) - \frac{h^2}{2} t } ].
\eeq
Therefore,
\beq
\frac{\ee[W_{N, t} ( \eta, \theta+h ) ] - \ee[ W_{N, t} ( \eta , \theta ) ] }{h} = \ee\left[ \tilde{W}_{N, t} ( \eta, \theta, B_0 ) \left( \frac{ \e^{ - h B_0 (t) -\frac{h^2 t}{2} } - 1 }{h} \right) \right]
\eeq
By dominated convergence (using Proposition \ref{prop:Wnt-bad-bd} together with Cauchy-Schwarz) the RHS converges to,
\beq
\lim_{h \to 0 } \ee\left[ \tilde{W}_{N, t} ( \eta, \theta, B_0 ) \left( \frac{ \e^{ - h B_0 (t) -\frac{h^2 t}{2} } - 1 }{h} \right) \right] = - \ee[ W_{N, t} ( \eta , \theta ) B_0 (t) ].
\eeq
On the other hand, by Proposition \ref{prop:dif-under} we have,
\beq
\lim_{h \to 0 } \frac{\ee[W_{N, t} ( \eta, \theta+h ) ] - \ee[ W_{N, t} ( \eta , \theta ) ] }{h}  = \ee\left[ \left( \del_\theta W_{N, t} \right) ( \eta , \theta ) \right].
\eeq
This yields the claim. \qed 

\remark As can be seen from the above proof, the derivative $\left( \del_\theta W_{N, t} \right) ( \eta , \theta)$ can be expressed as the directional (or Malliavin) derivative of $W_{N, t}$  obtained by perturbing  the Brownian motion $B_0$ in the direction $ - \int_0^\cdot  1_{[0,t]} \d s$.  The interested reader will find more information on the terminology ``Malliavin derivative" in \cite[Section 1.2]{malliavin}, but note that we will not be using any refined properties of these objects, other than the ones that can be obtained directly from the equations. \qed

\noindent{\bf Proof of Lemma \ref{lem:var-rep}}. Define
\beq
R_{N, t} ( \eta, \theta ) := \sum_{j=1}^N u_j ( t , \eta , \theta).
\eeq
Then,
\beq
\Var\left( W_{N, t} ( \theta, \theta ) \right) = \Var \left( R_{N, t} ( \theta , \theta ) \right) - t -2 \ee[ W_{N, t} ( \theta,  \theta) B_0 (t) ].
\eeq
The first line of \eqref{eqn:var-rep} now follows from Lemma \ref{lem:ibp} and the fact that,
\beq
\Var \left( R_{N, t} ( \theta , \theta ) \right) = N \Var ( u_1 (t , \theta, \theta ) ) = N \psiV_1 ( \theta ).
\eeq
 The second follows from the identity, 
 \beq
t- N \psiV_1 (\theta)  = \frac{ \d}{ \d \theta} ( t \theta - N \psiV_0 ( \theta) ) = \frac{ \d}{ \d \theta}  \ee[ W_{N, t} ( \theta , \theta ) ] = \ee[ \left( \del_\eta W_{N, t} \right) ( \theta , \theta ) ] + \ee[ \left( \del_{\theta} W_{N, t} \right) ( \theta, \theta ) ],
 \eeq
 where the differentiation under the integral is justified by Proposition \ref{prop:dif-under}. \qed
 
\subsection{First derivatives} \label{sec:aux-1d}

In this intermediate section we prepare  a few results about the signs of the first derivatives of the quantities involved in our proof. For the reader's convenience, we first restate Corollaries \ref{cor:1st-der-1} and \ref{cor:1st-der-2} as follows.

\bep \label{prop:1st-der-system}
The derivatives,
\beq
h_j (t, \eta, \theta)  := \left( \del_\theta u_j \right) ( t , \eta , \theta ) , \qquad k_j (t , \eta , \theta) := \left( \del_\eta u_j \right) ( t , \eta , \theta)
\eeq
exist and satisfy the systems of ODEs,
\begin{align}
\del_t h_1 (t)  &= - V'' (u_1 (t) ) h_1 (t) - 1 \notag\\
\del_t h_j (t) &= - V''(u_j (t) ) h_j (t) + V'' (u_{j-1} (t) ) h_{j-1} (t) , \qquad j \geq 2
\end{align}
and
\begin{align}
\del_t k_1 (t) &= - V'' (u_1 (t) ) k_1 (t)  \notag\\
\del_t k_j (t) &= - V'' (u_j  (t) ) k_j (t) + V'' (u_{j-1} (t) ) k_{j-1} (t) , \qquad j \geq 2
\end{align}
where we abbreviated $h_j (t) = h_j (t, \theta, \eta)$, etc. The initial data is $h_j (0) = 0$ for all $j$ and $k_j (0) \leq 0$.
\eep

Using the above we can easily derive the following (in fact, some of the signs appearing below were already stated in Corollaries \ref{cor:1st-der-1} and \ref{cor:1st-der-2}, but the idea of using the systems of ODEs satisfied by the various quantities to deduce monotonicity properties is important for our methods and so the repetition is warranted).
\bep
We that for all $ t \geq 0$  and all $j, n$  and all $\eta, \theta >0$ that,
\beq \label{eqn:1der-sign-1}
h_j (t , \theta , \eta ) \leq 0 , \qquad k_j ( t , \theta , \eta ) \leq 0.
\eeq
\beq\label{eqn:1der-sign-2}
\left( \del_\theta W_{N, t} \right) ( \eta, \theta ) \geq 0, \qquad \left( \del_\eta W_{N, t} \right) ( \eta , \theta ) \leq 0 .
\eeq
\eep
\proof The inequalities \eqref{eqn:1der-sign-1} follow from Lemmas \ref{lem:ode-1} and \ref{lem:ode-2} and the fact that $V'' (x) \geq 0$. The second inequality of \eqref{eqn:1der-sign-2} is immediate. For the first inequality we note that $w_j(t) := \del_\theta W_{j, t} ( \eta , \theta)$ satisfies the system,
\begin{align}
\del_t w_1 (t) &= - V''(u_1 (t) ) w_1 (t)  + V'' (u_1 (t) ) t \notag\\
\del_t w_j (t) &= - V''(u_j (t) ) w_j (t) + V'' (u_j (t) ) w_{j-1} (t),
\end{align}
with $0$ initial condition. The nonnegativity of the $w_j(t)$ then follows from Lemma \ref{lem:ode-4}. \qed

\bec \label{cor:bad-lower} 
We have the estimate,
\beq
\Var ( W_{N, t} ( \theta , \theta ) ) \geq |N \psiV_1 ( \theta) -t |.
\eeq
\eec
\proof This follows immediately from using the two representations of Lemma \ref{lem:var-rep} and the inequalities \eqref{eqn:1der-sign-2}. \qed

\subsection{Second derivatives} \label{sec:aux-2d}

In this section we prove the various auxilliary results used in the proof of Theorem \ref{thm: upper-bound} that rely on the second derivatives. We first state the following, whose proof is deferred to Appendix \ref{a:2-der-summary}.
\bep \label{prop:2nd-der-system}
The functions $u_j (t , \eta, \theta)$ are $C^2$ in the parameters $( \eta, \theta)$. Consider the system of ODEs for the functions $f_j (t)$ given the inhomogeneous terms $g_j (t)$,
\begin{align} \label{eqn:2nd-der-system}
\del_t f_1 (t ) &= - V'' (u_1  (t) ) f_1 (t) - V'''(u_1 (t) ) g_1 (t) \notag\\
\del_t f_j (t) &= - V'' (u_j (t) ) f_j(t) + V'' (u_{j-1} (t) ) f_{j-1} (t) \notag\\
& - V''' (u_j (t) ) g_j (t) +  V''' (u_{j-1} (t) ) g_{j-1} (t) ,
\end{align}
where $u_j (t) = u_j (t, \eta, \theta)$ as usual. 
The second deriviatives $\udd_j, \udi_j$ and $\uii_j$ all are solutions of this system, with $g_j = h_j^2, h_j k_j, k_j^2$, respectively, where the $k_j$ and $h_j$ are the derivatives described in Proposition \ref{prop:1st-der-system}. In each case, $g_j(t) \geq 0$ for all $t$ and furthermore the initial conditions for $\udd_j$ and $\udi_j$ are $\udd_j (0) = \udi_j (0) = 0$.
\eep

\subsubsection{Signs of second derivatives; proofs of Lemma \ref{lem:2nd-der-basic} and \ref{lem:2nd-ii}} \label{sec:aux-2nd-basic}

The entirety of this subsection is devoted to the proofs of Lemma \ref{lem:2nd-der-basic} and Lemma \ref{lem:2nd-ii}.  Consider first the case $w_j (t) = \left( \del_\theta^2 W_{j, t} \right) ( \eta , \theta )$ or $w_j (t) = \left( \del_\theta \del_\eta W_{j, t} \right) ( \eta , \theta)$. From Proposition \ref{prop:2nd-der-system} we see that these solve the system,
\begin{align}
\del_t w_j (t) = - V'' (u_j (t) ) w_j (t) + V'' (u_j (t) ) ) w_{j-1} (t) - V'''(u_j (t) ) g_j (t)
\end{align}
where the $g_j (t)$ are nonnegative and we defined $w_0 (t) = 0$. Since $V''' \leq 0$ we see that $w_j (t) \geq 0$ for all $t$ and $j$ by Lemma \ref{lem:ode-4}. 

We turn now to the claimed estimates for $\left( \del_\eta^2 W_{N, t} \right) ( \eta , \theta)$. Let us decompose
\beq
\uii_j (t , \theta , \eta ) = f_j (t) + m_j (t)
\eeq
where $f_j (t)$ satisfies the system \eqref{eqn:2nd-der-system} with $0$ initial condition and $g_j (t) = k_j(t)^2$. Then $m_j(t)$ is the solution to the homogeneous system of ODEs,
\begin{align}
\del_t m_1 (t) &= - V'' (u_1 (t) ) m_1 (t) \notag\\
\del_t m_j (t) &= - V'' (u_j (t) ) m_j (t) + V'' (u_{j-1} (t) ) m_{j-1} (t) 
\end{align}
with initial condition $m_j (0) = \uii_j (0)$. Then, arguing as above in the cases of $ \del_\theta^2 W_{j, t} $ and $\del_\eta \del_\theta W_{j, t}$ we see that
\beq
\sum_{j=1}^N f_j (t) \geq 0
\eeq
for all times $t$. On the other hand, by the last estimate of Lemma \ref{lem:ode-2} we see that,
\beq
\left| \sum_{j=1}^N m_j (t) \right| \leq \sum_{j=1}^N |\uii_j (0)| =: X ( \eta ).
\eeq
From this, we see that for any $\eta$ and $\theta$ that
\beq
\del_{\eta}^2 W_{N, t} ( \theta , \eta ) \geq  - X ( \eta ).
\eeq
From Proposition \ref{prop:1b} we see that,
\beq
\ee[  \sup_{ \eta \in I } |X ( \eta ) |^2 ] \leq C N^2
\eeq
for any  compact interval $I  \subseteq (0, \infty)$.  This completes the proof of Lemma \ref{lem:2nd-der-basic}.  Lemma \ref{lem:2nd-ii} now follows from the above estimate and that we have shown that for all compact intervals $I$ we have,
\beq
\inf_{ ( \eta , \theta ) \in I^2 } \left( \del_\eta^2 W_{N, t} \right) ( \eta , \theta ) \geq - \sup_{ \eta \in I } |X ( \eta ) |.
\eeq
This completes the proofs of Lemmas \ref{lem:2nd-der-basic} and \ref{lem:2nd-ii}. \qed

\subsubsection{Mixed partials; proof of Proposition \ref{prop:mixed}} \label{sec:aux-2nd-mixed}

The entirety of this subsection is devoted to the proof of Proposition \ref{prop:mixed}. The first estimate follows from \eqref{eqn:1der-sign-2}. Define,
\beq
A_j (t) := \left( \del_{\eta} \del_{\theta} W_{j, t} ( \eta , \theta ) + c_0 \del_{\theta} W_{j, t} ( \eta , \theta )\del_{\eta} W_{j, t} ( \eta , \theta ) \right) \bigg\vert_{\eta = \theta},
\eeq
where $c_0$ is the constant in the first inequality of \eqref{eqn:deriv-assump}. It suffices to prove that $A_n (t) \geq 0$ for all $n$ and $t$. This will be proven by induction. 

In the remainder of the proof we suppress the arguments $(\eta, \theta)$ in all of the functions considered. We will always consider these functions evaluated on the diagonal $\eta = \theta$, but the proof applies equally well off the diagonal.

We recall here that,
\begin{align}
\del_t \del_{\theta} W_{j, t} &= - V'' (u_j (t) ) h_j (t)  \notag\\
\del_t \del_\eta W_{j, t} &= - V'' (u_j (t) ) k_j (t) 
\end{align}
and
\beq
\del_t \del_{\eta} \del_\theta W_{j, t} = - V''( u_j (t) ) \udi_j (t) - V''' (u_j (t) ) k_j (t) h_j (t)
\eeq
which follow from Propositions  \ref{prop:1st-der-system} and Proposition \ref{prop:2nd-der-system}, respectively. Therefore,
\begin{align}\label{eqn: mixed partials 3}
    \frac{ \mathrm{d}}{ \mathrm{d} t} A_n(t) &= \frac{\mathrm{d}}{\mathrm{d}t}\left[c_0 \frac{\partial}{\partial \theta}W_{n, t} \frac{\partial}{\partial \eta}W_{n, t} + \frac{\partial^2}{\partial \theta \partial \eta}W_{n, t} \right] \nonumber \\
    &=-V''(u_n(t))\left[ c_0 k_{n}(t) \frac{\partial}{\partial \theta} W_{n, t} +c_0 h_n(t) \frac{\partial}{\partial \eta} W_{n, t} + \frac{\partial^2}{\partial \theta \partial \eta}u_n (s)\right] \nonumber \\
    & - V^{(3)}(u_n(t) )h_n(t)  k_n (t).
\end{align}
We now start the induction.  When $n=1$, we have $k_1 = \del_{\eta} W_{1, t}$ and $\del_\theta \del_\eta u_1 = \del_\theta \del_\eta W_{1, t}$.  Therefore, when $n=1$, we obtain from  \eqref{eqn: mixed partials 3} that,
\begin{equation}
\frac{\mathrm{d}}{\mathrm{d}t} A_1(t) =-V''(u_1 (t))A_1 (t) -\left[c_0 V''(u_1 (t))+V^{(3)}(u_1 (t))\right]h_1 (t) k_1 (t)\label{eqn: mixed partials 4}
\end{equation}
By the assumption \eqref{eqn:deriv-assump} and that $h_1$, $k_1$ are both negative by \eqref{eqn:1der-sign-1}, the  the last term of \eqref{eqn: mixed partials 4} is positive, giving
\beq   \frac{ \mathrm{d}}{ \mathrm{d} t } A_1 (t) \geq - V''(u_1(t) ) A_1 (t).\eeq
Since $A_1(0)=0$, we see that 
\beq
A_1 (t)\geq 0 ,
\eeq
for all $ t \geq 0$. 
Next, for $n>1$, we have,
\begin{align}
& k_n (t) \frac{ \del}{ \del \theta} W_{n, t} + h_n \frac{ \del}{ \del \eta} W_{n, t} = \frac{\del}{\del \eta} W_{n, t} \frac{\del}{\del \theta} W_{n, t}  - \frac{\del}{\del \eta} W_{n-1, t} \frac{\del}{\del \theta} W_{n, t}   + h_n\left( k_n + \frac{ \del}{ \del \eta} W_{n-1, t} \right) \notag\\
=& \frac{\del}{\del \eta} W_{n, t} \frac{\del}{\del \theta} W_{n, t}  - \frac{\del}{\del \eta} W_{n-1, t} \left( h_n +  \frac{\del}{\del \theta} W_{n-1, t}\right)   + h_n\left( k_n + \frac{ \del}{ \del \eta} W_{n-1, t} \right) \notag\\
=& \frac{\del}{\del \theta} W_{n, t} \frac{\del}{\del \eta} W_{n, t}  - \frac{\del}{\del \theta} W_{n-1, t}  \frac{\del}{\del \eta} W_{n-1, t}  + h_n (t) k_n (t) ,
\end{align}
In the first equality we substituted  $k_n = \del_\eta W_{n, t} - \del_\eta W_{n-1, t}$ and $\del_\eta W_{n, t} = k_n + \del_\eta W_{n-1, t}$. In the second equality we substituted $\del_\theta W_{n, t} = h_n + \del_{\theta} W_{n-1, t}$. 
Therefore, using the above equality in  \eqref{eqn: mixed partials 3} as well as $\del_\theta \del_\eta u_n = \del_\theta \del_\eta (W_{n, t} - W_{n-1, t})$ we find,
\begin{align}
\frac{\mathrm{d}}{\mathrm{d}t}A_n (t)&=-V''(u_n^T (t))\left[A_n(t)-A_{n-1}(t)\right]-(c_0 V''(u_n^T (t))+V^{(3)}(u_n^T (t)))h_n^T (t) k_n^T (t) \nonumber\\
&= -V''(u_n^T (t))A_n (t) \nonumber\\
& \quad+V''(u_n^T (t))A_{n-1}(t)-\left[c_0 V''(u_n^T (t))+V^{(3)}(u_n^T (t))\right]h_n^T (t) k_n^T (t).\label{eqn: mixed partials 5}
\end{align}
As in the case case $n=1$, since $k_n$ and $h_n$ are nonpositive by \eqref{eqn:1der-sign-1} and by \eqref{eqn:deriv-assump}, we have that the last term on the last line is positive. Since $V''(x) \geq 0$ and by the induction assumption that $A_{n-1} \geq 0$  we see that in fact the entire last line of \eqref{eqn: mixed partials 5} is positive so that
\beq
    \frac{\mathrm{d}}{\mathrm{d} t } A_n (t) \geq - V'' (u_n (t) ) A_n (t).\eeq
Using the initial condition $A_n (0) = 0$, we conclude the proof. \qed

\subsubsection{Variance comparison; proof of Lemma \ref{lem:var-comp}} 
 \label{sec:aux-var-comp}
The entirety of this subsection is devoted to the proof of Lemma \ref{lem:var-comp}. Let $\lambda > \theta$.  We have,
\begin{align}
\Var ( W_{N, t} ( \theta , \theta ) ) &= N \psiV_1 ( \theta ) - t + 2 \ee[ \left( \del_\theta W_{N, t} \right) ( \theta , \theta ) ] \notag\\
&\leq  N \psiV_1 ( \theta ) - t +2 \ee[ \left( \del_\theta W_{N, t} \right) ( \lambda , \lambda ) ] \notag\\
&= N \psiV_1 ( \theta ) - N \psiV_1 ( \lambda ) + \Var (W_{N, t} ( \lambda , \lambda ) ).
\end{align}
In the first and third lines we used the first representation in \eqref{eqn:var-rep}. In the second line we used that the function $(\eta , \theta ) \to \left( \del_\theta W_{N, t} \right) ( \eta, \theta)$ is increasing in both of its arguments by Lemma \ref{lem:2nd-der-basic}. Therefore,
\beq
\Var (W_{N, t} ( \theta, \theta ) ) - \Var ( W_{N, t} ( \lambda , \lambda ) ) \leq N ( \psiV_1 ( \theta ) - \psiV_1 ( \lambda ) ).
\eeq
For the lower bound using instead the second representation in \eqref{eqn:var-rep} we have,
\begin{align}
\Var ( W_{N, t} ( \theta , \theta ) ) &= t - N \psiV_1 ( \theta ) - 2 \ee[ \left( \del_\eta W_{N, t} \right) ( \theta, \theta ) ] \notag\\
&\geq t - N \psiV_1 ( \theta ) -2 \ee[ \left( \del_\eta W_{N, t} \right) ( \lambda , \lambda ) ] - C N ( \lambda - \theta )  \notag\\
&= N ( \psiV_1 ( \lambda ) - \psiV_1 ( \theta ) ) - C N ( \lambda - \theta ) - \Var ( W_{N, t} ( \lambda , \lambda ) ).
\end{align}
In the second line we used the second estimate of Lemma \ref{lem:2nd-der-basic} which results in the extra $CN ( \lambda - \theta)$ term compared to the upper bound. The claim now follows. \qed

\subsection{Proof of Corollary \ref{main-cor}} \label{sec:main-cor}

The entirety of this subsection is devoted to the proof of Corollary \ref{main-cor}. Fix $\theta >0$ and let $t_0 := N \psiV_1 ( \theta)$. If $t \geq t_0$ we may write,
\beq
W_{N, t}^\theta - \ee[ W_{N, t}^\theta ] = \left( \sum_{j=1}^N u_j (t , \theta , \theta ) - (B_0 (t) - B_0 (t-t_0) ) + \theta t_0  - \ee[ W_{N, t_0}^\theta] \right) + B_0 (t-t_0).
\eeq
The first term on the RHS has the same distribution as $W_{N, t_0}^\theta - \ee[W_{N, t_0}^\theta ]$ and so has variance $\O (N^{2/3} )$ by Theorem \ref{thm: upper-bound}. On the other hand, $B_0 (t - t_0)$ is a Gaussian random variable of variance $t-t_0 \gg N^{2/3}$. 

If $t \leq t_0$ we note that $W_{N, t}^\theta$ has the same distribution as the random variable,
\beq
Z := \left( \sum_{j=1}^N u_j ( t + t_0, \theta , \theta ) - (B_0 (t+t_0) - B (t_0) ) + \theta t \right).
\eeq
We then decompose,
\beq
Z - \ee[Z] = \left( \sum_{j=1}^N u_j ( t + t_0, \theta , \theta )  - (B_0 (t+t_0) - B (t) ) + \theta t_0 - \ee[ W_{N, t_0}^\theta ] \right) + B (t_0) - B(t).
\eeq
The first term on the RHS has the same distribution as  $W_{N, t_0}^\theta - \ee[W_{N, t_0}^\theta ]$ and so has variance $\O (N^{2/3} )$ by Theorem \ref{thm: upper-bound}, and the second term is a Gaussian with variance $|t_0 -t|$. The claim follows. \qed
\section{Pseudo-Gibbs measures and concentration}
 \label{sec:psg}

This section is devoted to the introduction of the pseudo-Gibbs measures and proving tail estimates of the ``first jump'' time on the scale $N^{2/3}$ with respect to the annealed measure. In the O'Connell-Yor case, these estimates results can be interpreted as an analog of Proposition 3.3 of \cite{EJS} for exponential last passage percolation. 

Before introducing the pseudo-Gibbs measure, we will first give an exact calculation of a generating function in Section \ref{sec:generating}. Then, we will introduce the pseudo-Gibbs measures in Section \ref{sec:psg-2} and as well as give them a dynamical interpretation. Finally, in Section \ref{sec:psg-conc} we prove the advertised tail estimates.

\subsection{The Rains-EJS generating function} \label{sec:generating}
In this section we derive the analog in our setting of a generating function considered by Rains
\cite{rains} and Emrah-Janjigian-Seppalainen \cite{EJS}  for Last Passage Percolation with exponential weights on $\mathbb{Z}^2$ (the discrete, ``zero temperature'' version of the O'Connell-Yor polymer).

\bep \label{prop:generat}
Let $W_{N, t} ( \eta, \theta)$ be as in \eqref{eqn:W-def} and define,
\beq
\varphi ( \theta ) := N \psiV_{-1} ( \theta ) - \frac{1}{2} \theta^2 t = N \log (Z( \theta)) - \frac{1}{2} \theta^2 t.
\eeq
Then,
\beq
\ee\left[ \exp \left(  ( \eta - \theta ) W_{N, t} ( \eta , \theta ) \right) \right] = \exp \left( \varphi ( \theta ) - \varphi ( \eta ) \right).
\eeq
\eep
\proof We write,
\beq
\ee\left[ \exp \left(  ( \eta - \theta ) W_{N, t} ( \eta , \theta ) \right) \right] = \ee \left[ \exp \left\{  ( \eta  - \theta) (W_{N, t} ( \eta , \theta ) - W_{0, t} ( \eta , \theta)) \right\} \exp \left\{ ( \eta - \theta ) W_{0, t} ( \eta , \theta ) \right\} \right]
\eeq
where,
\beq
W_{0, t} ( \eta , \theta ) := - B_0 (t) + \theta t.
\eeq
Since,
\beq
W_{N, t} ( \eta , \theta ) - W_{0, t} ( \eta , \theta) = \sum_{j=1}^N u_j ( t , \eta , \theta)
\eeq
and
\begin{align}
( \eta - \theta) W_{0, t} ( \eta , \theta ) &= ( \theta - \eta) B_0 (t) - \frac{1}{2} ( \theta  - \eta )^2 t \notag\\
&+ \frac{1}{2} ( \eta^2 - \theta^2 ) t,
\end{align}
we may apply the Cameron-Martin theorem to find,
\begin{align}
\ee\left[ \exp \left(  ( \eta - \theta ) W_{N, t} ( \eta , \theta ) \right) \right] &= \ee \left[ \exp \left( ( \eta-  \theta ) \sum_{j=1}^N u_j (t , \eta , \theta ) \right) \e^{  ( \theta - \eta) B_0 (t) - \frac{1}{2} ( \theta - \eta )^2 t } \right] \e^{ \frac{1}{2} ( \eta^2 - \theta^2 ) t } \notag\\
&= \ee \left[ \exp \left( ( \eta - \theta ) \sum_{j=1}^N u_j ( t , \eta , \eta ) \right) \right] \e^{ \frac{1}{2} ( \eta^2 - \theta^2 ) t } \notag\\
&= \frac{ ( Z( \theta )^N}{ Z ( \eta )^N} \e^{ \frac{1}{2} ( \eta^2 - \theta^2 ) t } ,
\end{align}
where in the last equality we used that the $u_j (t, \eta, \eta)$ are distributed as iid $\nu_\eta$, this being the invariant measure. 
The claim now follows from the definition of $\varphi$. \qed

\subsection{Pseudo-Gibbs measure}
\label{sec:psg-2}
For any $(N, t, \eta, \theta)$ let us define the measure $\Eet_{N, t}$ on $[0, t]$ by its action on, say, bounded measureable $F: [0, t] \to \rr$ via,
\beq \label{eqn:psg-def}
\Eet_{N, t} [F] := \int_{0 < s_0 < \dots < s_{N-1} < t } \exp \left( - \sum_{j=0}^{N-1} \int_{s_j}^{s_{j+1} } V'' (u_j ( u) ) \d u \right) F(s_1)  \prod_{j=0}^{N-1} V'' (u_{j+1} (s_j ) ) \d s,
\eeq
where we use the convention $s_{N} = t$. 
 We will only need a few specific choices of (deterministic) $F$ in this paper which are all continuous or piecewise continuous and bounded.  For $f : \rr_+ \to \rr$ piecewise continuous, consider the solution $\hf_j (t, \eta , \theta)$ to the following system of ODEs with initial condition $\hf_j (0) = 0$ for all $j$,
 \begin{align} \label{eqn:hf-def}
 \del_t \hf_1 (t) &= - V''(u_1 (t) ) \hf_1 (t) - f(t) \notag\\
 \del_t \hf_j (t) &= - V'' (u_j (t) ) \hf_j (t) + V'' (u_{j-1} (t) ) \hf_{j-1} (t) , \qquad j \geq 2.
 \end{align}
 We abbreviated $u_j (t) = u_j ( t , \eta , \theta)$ and $\hf_j (t ) = \hf_j (t , \eta , \theta)$.
 
 The connection between these two objects is as follows.
 \bep \label{prop:psg}
 Let $F : [0, t] \to \rr $ be of the form,
 \beq
 F(s) = \int_0^s f(u) \d u
 \eeq
 for piecewise continuous $f$. Then, 
 \beq \label{eqn:psg-eq-1}
 \Eet_{N, t} [F] = \sum_{j=1}^N \hf_j ( t , \eta , \theta ) + F(t).
 \eeq
 Moreover, for all $F$,
 \beq \label{eqn:l-inf-bd}
  | \Eet_{N, t} [ F ] | \leq \| F \|_\infty
 \eeq
 and $\Eet_{N, t}$ defines a positive measure that assigns to $[0, t]$ weight less than or equal to $1$.  If $f$ is nonnegative, then
 \beq
 \hf_j (t) \leq 0
 \eeq
 for all $j$ and $t$.
 \eep
 \proof In the proof we suppress explicit dependence on $(\eta , \theta)$ of the various quantities involved where convenient. By definition we have,
 \begin{align} \label{eqn:psg-proof-1}
 \Eet_{1, t} [F]&= \int_0^t \e^{ - \int_{s_0}^t V'' (u_1 (u) ) \d u } V'' (u_1 (s_0) ) F (s_0) \d s_0 , \notag\\
 \Eet_{n, t}[F] &= \int_0^t \e^{ - \int_{s_{n-1}}^t V'' (u_n (u) ) \d u } V'' (u_n (s_{n-1} ) ) \Eet_{n-1, s_{n-1}} [F] \d s_{n-1}, \qquad n \geq 2 
 \end{align}
 Denote temporarily,
 \beq
 \tilE_{n, t} := \sum_{j=1}^n \hf_j (t) + F(t).
 \eeq
 Then,
 \begin{align}
 \del_t \tilE_{n, t} = - V'' ( u_n (t) ) \tilE_{n, t} + V'' ( u_n (t) ) \tilE_{n-1, t}
 \end{align}
 where $\tilE_{0, t} := F(t)$. The solution to this system of ODEs is given by \eqref{eqn:psg-proof-1}, proving the desired equality. Temporarily denote now $\hatE_{n, t} := \Eet_{n, t} [1]$. We prove by induction that $\hatE_{n, t} \leq 1$. Clearly by \eqref{eqn:psg-proof-1} all the $\hatE_{N, t}$ are nonnegative since $V$ is convex. We have,
 \beq
 \hatE_{1, t} = \int_0^t \e^{ - \int_{s_0}^t V'' (u_1 (u) ) \d u } V'' (u_1 (s_0) ) \d s_0 = 1 - \e^{ - \int_0^t V'' (u_1 (u) ) \d u } \leq 1.
 \eeq 
 Assuming that $\hatE_{n-1, t} \leq 1 $ have,
 \begin{align}
 \hatE_{n, t} &= \int_0^t \e^{ - \int_{s_{n-1}}^t V'' (u_n (u) ) \d u } V'' (u_n (s_{n-1} ) ) \hatE_{n-1, t} \d s_{n} \notag\\ 
 &\leq \int_0^t \e^{ - \int_{s_{n-1}}^t V'' (u_n (u) ) \d u } V'' (u_n (s_{n-1} ) )  \d s_{n-1} \notag\\
 &= 1 - \e^{ - \int_0^t V'' (u_n (u ) ) \d u } \leq 1
 \end{align}
 which completes the proof that the $\hatE_{n, t} \leq 1$ for all $n$ and $t$.  Finally, the nonpositivity of the $\hf_j (t)$ in the case of nonnegative $f$ follows from Lemma \ref{lem:ode-1}. \qed
 
\remark The functions $\hf_j (t)$ can be obtained by adding the parameter $-\delta f$ to the RHS of the equation for $u_1 (t) $ in \eqref{eqn: system}, differentiating with respect to $\delta$ and setting $\delta = 0$. In the language of Malliavin derivatives discussed above in the context of the proof of Lemma \ref{lem:ibp}, the $\hf_j$ are the Malliavin derivatives of the $u_j$ obtained by perturbing the Brownian motion $B_0$ in the direction $- \int_0^\cdot f (s) \d s.$ For example, when $f = 1$, one obtains that $\hf_j = \del_\theta u_j $. \qed
 
 \subsection{Upper tail bound with respect to Pseudo-Gibbs measure} \label{sec:psg-conc}

 \subsubsection{Preliminaries}
 
 In order to prove our upper tail estimate, we will need to change some of the parameters in the measure $\Eet_{N, t}$ in order to apply Proposition \ref{prop:generat}. The following two lemmas establish the required monotonicity properties which are then applied to the observable of interest in Corollary \ref{cor:psg-mono}. 
 
 \bel
 Let $F$ be of the form,
 \beq
 F(s) = \int_0^s f(u) \d u
 \eeq
 for piecewise continuous nonnegative $f$. Then, the function,
 \beq
 \theta \to \Eet_{N, t} [F]
 \eeq
 is differentiable in $\theta$ and nondecreasing.
 \eel
 \proof Differentiability of $\Eet_{N,t} [F]$ follows easily from its definition as an iterated integral and the fact that $V$ is smooth and the $u_j (t, \eta, \theta)$ are differentiable. 
 
 We now turn to proving the monotonicity. We temporarily denote the derivatives by $v_j (t) := \del_\theta \Eet_{N, t} [F]$ and denote $u_j (t) = u_j ( t , \theta , \eta )$. They satisfy,
 \begin{align}
 \del_t v_j (t) = - V''(u_j (t) ) v_j (t) + V'' (u_j (t) ) v_{j-1} (t) - V''' (u_j (t) ) ( \hf_j (t) ) ( \del_\theta u_j (t) )
 \end{align}
 with the convention that $v_{0} (t) = 0$.  By \eqref{eqn:1der-sign-1} and Proposition \ref{prop:psg} we have,
 \beq
 - V''' (u_j (t) ) ( \hf_j (t) ) ( \del_\theta u_j (t) ) \geq 0
 \eeq
 for all $t$. From Lemma \ref{lem:ode-4} we have that $v_j (t) \geq 0$ for all $t$. The claim follows. \qed
 
 \bel\label{lem: pseudo-monotone}
 Let $V$ be of O'Connell-Yor type and let $0 < a \leq c_0$ where $c_0$ is the constant from \eqref{eqn:deriv-assump}. Let $F$ be of the form,
 \beq
F(s) = \int_0^s f(u) \d u 
 \eeq
 for nonnegative piecewise continuous $f$. 
  Then, the function
 \beq
 \eta \to \e^{ a W_{N, t} ( \eta , \theta ) } \Eet_{N, t} [F]
 \eeq
 is increasing.
 \eel
 \remark In the O'Connell-Yor case, one can take $a=1$ and in fact the quantity under consideration is independent of $\eta$. Indeed, in this case the prefactor cancels the appearance of partition function in \eqref{eqn:oy-gibbs} and all that remains is an integral restricted to $s_0 \geq 0$ which does not depend on $\eta$. \qed
 
 \proof By differentiation (differentiability follows as in the previous lemma), it suffices to prove that
 \beq
 a  \left( \del_\eta W_{N, t} \right) ( \eta , \theta )  \Eet_{N, t} [F] + \del_\eta \Eet_{N, t} [F] \geq 0.
 \eeq
The argument will be similar to the proof of Proposition \ref{prop:mixed}. We denote the quantity on the right in the above display by $A_N (t)$ and suppress dependence of other quantities in the proof on the parameters $(\eta , \theta)$ where convenient. Recall the notation $k_j (t) = \del_\eta u_j ( t , \eta , \theta)$. We have,
\begin{align}
\del_t \del_\eta \Eet_{n, t} [F] &= -V'' (u_n (t) ) \del_\eta \hf_n (t) - V'''(u_n) \hf_n (t) k_n  (t) \notag\\
&= - V'' (u_n (t) ) \del_\eta \Eet_{n, t} [F]  + V'' (u_n (t) ) \del_\eta \Eet_{n-1,t} [F] - V''' (u_n (t) ) \hf_n (t) k_n (t) ,
\end{align}
where we use the convention $\Eet_{0, t} [F] := F(t)$ (and so $\del_\eta \Eet_{0, t} [F]  = 0$).  Similarly,
\begin{align}
\del_t \left\{ \left(  \del_\eta W_{n, t} \right) \Eet_{n, t} [F] \right\} =& - V''(u_n (t) ) ( \del_\eta W_{n, t}  ) \hf_n (t) - V'' (u_n (t) ) k_n(t) \Eet_{n, t} [F] \notag\\
=& - V'' (u_n ) ( \del_\eta W_{n, t}  )\Eet_{n, t}[F]  \notag\\
+&  V'' (u_n) ( \del_\eta W_{n-1, t}  ) \Eet_{n-1, t}[F] - \hf_n (t) V'' (u_n) k_n (t) ,
\end{align}
where we use the convention $W_{0 , t} ( \eta , \theta ) = \theta t$. 
Therefore,
\begin{align}
\del_t A_n(t) &= - V'' (u_n (t) ) A_n (t) + V'' (u_n (t) ) A_{n-1} (t)  \notag\\
&+ (\hf_n(t) ) k_n(t) ( - V''' (u_n (t) ) - a V'' (u_n (t) ) ).
\end{align}
where $A_0 (t) = 0$.  By the assumption on $a$ and the fact that $k_n (t) \leq 0$ by \eqref{eqn:1der-sign-1} and $\hf_n (t) \leq 0$ by Proposition \ref{prop:psg} we have that the last line above is a positive function. The nonnegativity of the $A_n(t)$ now follows from Lemma \ref{lem:ode-4}. \qed 

\bec \label{cor:psg-mono}
Let $ 0 \leq t_0 \leq t$ and let,
\beq
F(s) = \1_{ \{ s \geq t_0 \} }.
\eeq
Let $0 < a \leq c_0$ where $c_0$ is from \eqref{eqn:deriv-assump}. 
Then, the functions
\beq
\theta \to \Eet_{N, t} [F]
\eeq
and
\beq
\eta \to \e^{ a W_{N, t} ( \eta  , \theta ) } \Eet_{N, t} [F]
\eeq
are nondecreasing.
\eec
\proof Let 
\beq
f_n (u) := n \1_{ \{ t_0 \leq u \leq t_0 + n^{-1} \} } 
\eeq
and
\beq
F_n(s) := \int_0^s f_n (u) \d u.
\eeq
From the definition \eqref{eqn:psg-def} we see that
\beq
\Eet_{N, t} [F] = \lim_{n \to \infty} \Eet_{N, t} [ F_n ].
\eeq
The previous two lemmas apply to $F_n$ and so the claim follows. \qed

\subsubsection{Tail estimates}

We now turn to the proof of the moderate deviation estimates of the first jump time with respect to the pseudo-Gibbs measure in the equilibrium case $\eta = \theta$. We first derive an estimate for the probability that the first jump time is positive, and then use stationarity to deduce a general tail estimate from the first case. 
\bep \label{prop:conc-1}
Let $\theta_0 >0$ be such that,
\beq
\psiV_{2} ( \theta_0 ) < 0.
\eeq
Let $\eps_0 >0$ be,
\beq
\eps_0 := \frac{1}{2} \left( c_0 \wedge \theta_0 \right),
\eeq
where $c_0$ is the constant from Assumption \ref{eqn:deriv-assump}. 
There is a constant $C_0 >0$ that depends only on
\beq
\sup_{ |\theta - \theta_0 | \leq \eps_0} | \psiV_3 (\theta ) |
\eeq
so that for $t = N \psiV_1 ( \theta_0)$ we have for all $\theta$ satisfying $ 0 < \theta_0 - \theta < \eps_0$ that
\beq
\ee \left[ E^{(\theta, \theta)}_{N, t} [ \1_{\{ s_0 > 0 \}} ] \right] \leq \exp\left( N \frac{\psiV_2 (\theta_0)}{16} (\theta_0 - \theta)^3 + N C_0 (\theta_0 - \theta)^4 \right)
\eeq
\eep
\proof Let $4 a = \theta_0 - \theta > 0$. Define also,
\beq
\lambda := \theta + 2a .
\eeq
By Corollary \ref{cor:psg-mono} we have,
\begin{align}
\Ett_{N, t} [\1_{ \{ s > 0 \} } ] &\leq E^{( \theta, \lambda )}_{N, t} \left[ \1_{ \{ s > 0 \} } \right] \notag\\
&\leq \e^{a ( W_{N, t} ( \theta_0 , \lambda ) - W_{N, t} ( \theta , \lambda ) ) } E^{( \theta_0, \lambda ) }_{N, t}  \left[\1_{ \{ s > 0 \} } \right] \notag\\
&\leq \e^{a ( W_{N, t} ( \theta_0 , \lambda ) - W_{N, t} ( \theta , \lambda ) ) } ,
\end{align}
 where we used \eqref{eqn:l-inf-bd} in the last line. By Cauchy-Schwarz and Proposition \ref{prop:generat},
 \begin{align}
 \ee \left[ \e^{a ( W_{N, t} ( \theta_0 , \lambda ) - W_{N, t} ( \theta , \lambda ) ) }  \right]^2 &\leq \ee \left[ \e^{ 2a  W_{N, t} ( \theta_0 , \lambda ) } \right]\left[ \e^{ -2a  W_{N, t} ( \theta , \lambda ) } \right] \notag\\
 &= \ee \left[ \e^{ (\theta_0 -\lambda)  W_{N, t} ( \theta_0 , \lambda ) } \right]\left[ \e^{ ( \theta - \lambda)  W_{N, t} ( \theta , \lambda ) } \right] \notag\\
 &= \exp \left( N (2 \psiV_{-1} ( \lambda )  - \psiV_{-1} ( \theta_0 )  - \psiV_{-1} ( \theta )  ) + \frac{t}{2} (   \theta_0^2 + \theta^2 - 2 \lambda^2 ) \right).
 \end{align}
 A Taylor expansion at $\theta_0$ gives
 \begin{align}
 2 \psiV_{-1} ( \lambda )  - \psiV_{-1} ( \theta_0 )  - \psiV_{-1} ( \theta )  ) = - 4 \psiV_1 ( \theta_0 ) a^2 + 8 \psiV_2 ( \theta_0 ) a^3 + \O ( a^4 ).
 \end{align}
On the other hand,
\beq
 \theta_0^2 + \theta^2 - 2 \lambda^2 = 8 a^2.
\eeq
We conclude using $t = N \psiV_1 ( \theta_0)$. \qed

\bel \label{lem:stat}
Let $0 < \tau < t$. We have the equality in distribution,
\beq
\Ett_{N, t} [ \1_{ \{ s_1 > \tau \} } ] \stackrel{d}{=} \Ett_{N, t - \tau } [ \1_{ \{ s_1 > 0 \} } ].
\eeq
\eel
\proof By the definition \eqref{eqn:psg-def} we have,
\begin{align}
\Ett_{N, t} [ \1_{ \{ s > \tau \} } ] &= \int_{0 < s_0 < \dots < s_{N-1} < t } \exp \left( - \sum_{j=0}^{N-1} \int_{s_j}^{s_{j+1} } V'' (u_j ( u) ) \d u \right) \1_{ \{ s_0> \tau \} }  \prod_{j=1}^N V'' (u_j (s_{j-1} ) ) \d s \notag\\
&=  \int_{\tau < s_0 < \dots < s_{N-1} < t } \exp \left( - \sum_{j=0}^{N-1} \int_{s_j}^{s_{j+1} } V'' (u_j ( u) ) \d u \right) \1_{ \{ s_0 > \tau \} }  \prod_{j=1}^N V'' (u_j (s_{j-1} ) ) \d s \notag \notag\\
=  \int_{0 < s_0 < \dots < s_{N-1} < t-\tau } &\exp \left( - \sum_{j=0}^{N-1} \int_{s_j}^{s_{j+1} } V'' (u_j ( u+\tau) ) \d u \right) \1_{ \{ s_0 > 0 \} }  \prod_{j=1}^N V'' (u_j (s_{j-1} +\tau) ) \d s
\end{align}
Now, $\{ u_j ( s+\tau ) \}_{1 \leq j \leq N , 0 \leq s \leq t}$ has the same distribution as $\{ u_j ( s ) \}_{ 1 \leq j \leq N , 0 \leq s \leq t }$ since the solution to \eqref{eqn: system} is a Markov process and we are assuming that the distribution of $\{ u_j (0) \}_{j=1}^N$ is invariant.  Therefore, the last line has the same distribution as,
\beq
\int_{0 < s_0 < \dots < s_{N} < t-\tau } \exp \left( - \sum_{j=0}^{N-1} \int_{s_j}^{s_{j+1} } V'' (u_j ( u) ) \d u \right) \1_{ \{ s_0 > 0 \} }  \prod_{j=1}^N V'' (u_j (s_{j-1} ) ) \d s = \Ett_{N, t - \tau } \left[ \1_{ \{ s_0 > 0 \} } \right].
\eeq 
This yields the claim. \qed

The following is the main technical result of this section, and gives an upper tail estimate for the first jump time under the pseudo-Gibbs measure. We remark here that the upper bound, Theorem \ref{thm: upper-bound} also follows from the following result, given the variance representation on the first line of \eqref{eqn:var-rep}. However, this approach requires the additional assumption $\psiV_2 ( \theta ) <0$.

\bep \label{prop:conc-2}
Let $I_0$ be a compact interval supported in $(0, \infty)$ on which $\psiV_2 ( \theta ) < 0$ for all $ \theta \in I_0$. Then there is a $\delta >0$ and $C_0 >0$ so that for all pairs $(t , \theta)$ satisfying $\theta \in I_0$ and
\beq
|t-N \psiV_1 ( \theta ) | \leq \delta N
\eeq
we have for all $0 \leq w \leq \delta N$ that,
\beq
\ee\left[ \Ett_{N, t} [ \1_{ \{ s_0 > e (\theta, t) + w \} } ] \right] \leq \exp \left( - \frac{ w^3}{16 N^2  \psiV_2 ( \theta)^2 } + C_0 N^{-3} w^4 \right)
\eeq
where,
\beq
e ( \theta , t ) := t - N \psiV_1 ( \theta ).
\eeq
as long as $e ( \theta ,t ) + w \geq 0$.
\eep
\proof Let $\delta_1 >0$ be such that $\psiV_2 ( \theta ) < 0$ for all $ \theta \in I_1 := I_0 + [-\delta_1, \delta_1 ]$ and this interval remains strictly contained in $(0, \infty)$. Let us take $\delta >0$ so small that for all $0 \leq w_0 \leq 10 \delta$ and all $\theta \in I_0$, the equation
\beq
\psiV_1 ( \theta ) - \psiV_1 ( v  ) = w_0.
\eeq
 has a (necessarily unique) solution $ v \in I_1$ admitting the expansion,
 \beq
 (\theta - v) = - \frac{ w_0}{ \psiV_2 ( \theta ) }  + \O ( w_0^2 ) ,
 \eeq
 and that
 \beq
 | \theta - v | \leq \frac{1}{10} \left(  c_0 \wedge \left( \inf I_1 \right) \right)
 \eeq
 where $c_0$ is the constant in \eqref{eqn:deriv-assump}. 
Now, given $\theta \in I_0$ and $t$ and $w$ satisfying $|t - N \psiV_1 ( \theta ) | \leq \delta N$ and $0 \leq w \leq \delta N$ let $\nu \in I_1$ satisfy,
\beq
\psiV_1 ( \theta ) - \psiV_1 ( \nu ) = \frac{w}{N} .
\eeq 
The claim is vacuous if $e(\theta, t) + w \geq t$ (the LHS of the desired inequality is then $0$), so assume that $0 \leq e(\theta, t) + w \leq t$ (alternatively we could also reduce the value of $\delta >0$ to enforce this). 
By Lemma \ref{lem:stat} we have,
\beq
\ee \left[ \Ett_{N, t} [ \1_{ \{ s_0 > e ( \theta , t) + w \} } ] \right] = \ee\left[ \Ett_{N, t - e(\theta, t) - w } [ \1_{ s_0 > 0 } ] \right]
\eeq
Now by definition,
\beq
t - e ( \theta ,t ) - w = \psiV_1 ( \nu ).
\eeq
We may apply Proposition \ref{prop:conc-1} with the $\theta_0$ there equal to the $\nu$ here and the $\theta$ there equal to the $\theta$ here. Therefore,
\begin{align}
\ee[ \Ett_{N, t - e(\theta, t) - w } [ \1_{ s_0 > 0 } ] ] &\leq \exp \left( N \frac{ \psiV_2 ( \nu )}{16} ( \nu - \theta )^3 + N C ( \nu - \theta )^4 \right) \notag\\
& \leq \exp \left( - \frac{ w^3}{16 N^2  \psiV_2 ( \theta)^2 } + C_0 N^{-3} w^4 \right)
\end{align}
and the claim follows. \qed

Finally, we derive a corollary that is suited for application in the next section.

\bec \label{cor:conc}
Let $\theta_0 >0$ be such that $\psiV_2 ( \theta_0) < 0$.  Let
\beq
t_0 := N \psiV_1 ( \theta_0), \qquad \lambda := \theta_0 + N^{-1/3}.
\eeq
There are $C>0$ and $c>0$ so that for all $Y>0$ we have,
\beq
\ee[ E^{ ( \lambda, \lambda ) }_{N, t_0} [ (s_0 - Y N^{2/3} )_+ ] ] \leq C \left(  N^{2/3} \e^{ - c (Y-C)_+^3} +  \e^{ - c N } \right).
\eeq
\eec
\proof Note that $e ( \lambda , t) = \O ( N^{2/3} )$. By Proposition \ref{prop:conc-2} there is an $\eps_1 >0$ so that,
\beq
\ee[ E^{ ( \lambda \lambda ) }_{N, t_0} [ \1_{ \{ s_0 > \eps_1 N/2 \} } ] ] \leq C \e^{ - c \eps_1^3 N }
\eeq
and for all $0 < y < \eps_1 N$ we have,
\beq \label{eqn:conc-2-a1}
\ee[ E^{ ( \lambda, \lambda ) }_{N, t_0} [ \1_{ \{ s_0 > y \} } ] ] \leq C \e^{ - c N^{-2} (y-C N^{2/3} )^3_+ }.
\eeq
The measure defined by $E^{ ( \lambda, \lambda ) }_{N, t_0}$ is supported in the interval $[0, t_0]$ and so we have almost-everywhere with respect to to the pseudo-Gibbs measure that,
\beq
(s_0-YN^{2/3} )_+ \leq t_0 \leq C N.
\eeq
Therefore,
\beq
\ee[ E^{ ( \lambda, \lambda ) }_{N, t_0} [ (s_0 - Y N^{2/3} )_+ \1_{  \{s_0 > \eps_1 N /2 \} } ] ] \leq  C N \ee[ E^{ ( \lambda \lambda ) }_{N, t_0} [  \1_{  \{s_1 > \eps_1 N /2 \} } ] ] \leq C   \e^{ - c N}
\eeq
for some $c, C>0$.  On the other hand,
\begin{align}
\ee[ E^{ ( \lambda, \lambda ) }_{N, t_0} [ (s_0 - Y N^{2/3} )_+  \1_{ \{ s_0 < \eps_1 N /2 \} }]] &= \int_0^\infty \ee[ E^{ ( \lambda, \lambda ) }_{N, t_0} [ \1_{ \{ Y N^{2/3} + u < s_0 < \eps_1 N/2 \} }] ] \d u \notag\\
&=  N^{2/3}\int_0^\infty \ee[ E^{ ( \lambda, \lambda ) }_{N, t_0} [ \1_{ \{ (Y+u)N^{2/3} < s_0 < \eps_1 N/2 \} }] ] \d u \notag\\
&\leq N^{2/3} \int_0^\infty C \e^{ - c (Y+u - C)_+^3 } \d u \notag\\
&\leq  CN^{2/3} \e^{ - c (Y - C)_+^3 }
\end{align}
In the first inequality we used \eqref{eqn:conc-2-a1}. 
This yields the claim. \qed

\section{Lower bound} \label{sec:lower}

In this section we give the proof of our lower bound for the variance, Theorem \ref{thm:lower}. We first introduce a three-parameter version of the height function $W_{N, t} ( \eta , \theta)$ denoted by $\tilW_{N, t} ( \eta , \theta_1 , \theta_2)$ below. The role of $\eta$ as the initial data parameter will remain the same. We separate the drift or driving parameter into two regimes, when $s_1 \lesssim N^{2/3}$ and when $s_1 \gg N^{2/3}$. The first is the regime of ``strong dependence'' and the second of ``weak dependence.'' The weak dependence of the three-parameter $\tilW_{N, t} ( \eta , \theta_1, \theta_2)$ on its last argument is quantified by the concentration estimates of Corollary \ref{cor:conc}. 

Let us remark here that all questions of well-posedness and differentiability of the three-parameter system are almost identical to the two-parameter one. Further discussion is deferred to Appendix \ref{a:three}.

\subsection{Three-parameter height function}
\label{sec:three-param}
We introduce a three-parameter height function  $\tilW_{N, t} ( \eta, \theta_1, \theta_2)$ as follows. Fix a $ Y>0$, This parameter will later be chosen to be large depending on some other constants that appear in the proof of Proposition \ref{prop:lower} below. Then, define $\tilu_j ( t , \eta , \theta_1 , \theta_2 )$ as the solution to, 
\begin{align} \label{eqn:3-sys}
\d \tilu_1 ( t , \eta , \theta_1, \theta_2 ) &= - V'(\tilu_1 ( t , \eta , \theta_1, \theta_2 )) \d t+ \d B_0 + \d B_1 - \theta_1 \1_{  \{ t \in [0, Y N^{2/3} ] \} } \d t - \theta_2 \1_{\{  t> Y N^{2/3} \} } \d t \notag\\
\d \tilu_j ( t , \eta , \theta_1 , \theta_2 ) &= -V' ( \tilu_j ( t , \eta , \theta_1 , \theta_2 ) ) + V' ( \tilu_{j-1} ( t , \eta , \theta_1 , \theta_2 ) ) - \d B_{j-1} + \d B_j ,
\end{align}
with initial data equal to $\tilu_j ( 0 , \eta , \theta_1 ,\theta_2 ) = H_\eta ( q_i )$.  That is, we set the drift parameter equal to $\theta_1$ for time up to $Y N^{2/3}$ and then $\theta_2$ for all later times $t$. Then, $\tilW_{N, t}$ is defined by,
\begin{align} \label{eqn:3-height}
\tilW_{N, t} ( \eta , \theta_1, \theta_2 ) &:= \sum_{j=1}^N \tilu_j ( t , \eta , \theta_1 , \theta_2 ) - B_0 (t) \notag\\
+& \int_0^t \theta_1 \1_{  \{ s \in [0, Y N^{2/3} ] \} } \d t + \theta_2 \1_{\{  s> Y N^{2/3} \} } \d s
\end{align}
The point of this definition is that, due to the concentration estimates of the previous section, the height function depends strongly only on $\theta_1$ and only weakly on $\theta_2$ for $Y$ large enough. 

The next  lemma establishes some monotonicity properties similar to the case of the two-parameter height function. 
\bel \label{lem:three-param-1}
For any $\eta, \theta_1, \theta_2 >0$ we have,
\beq
\left( \del_\eta \del_{\theta_2} \tilW_{N, t} \right) ( \eta , \theta_1 , \theta_2 ) \geq 0 , \qquad  \left(\del_{\theta_2}^2  \tilW_{N, t} \right) ( \eta , \theta_1 , \theta_2 ) \geq 0 ,
\eeq
for all $N$ and $t \geq 0$.
\eel
\proof For notational convenience we drop dependence of the various quantities on the arguments $(\eta, \theta_1, \theta_2)$. Define $\tilh_j (t) := \del_{\theta_2} \tilu_j (t)$ and $\tilk_j (t) := \del_\eta \tilu_j (t)$. The $\tilh_j$ obey,
\begin{align} \label{eqn:tilh-system}
\del_t \tilh_1 (t) &= - V'' (\tilu_1 (t) ) \tilh_1 (t) - \1_{ \{ t \geq Y N^{2/3} \} } \notag\\
\del_t \tilh_j (t) &= - V'' ( \tilu_j (t) ) \tilh_j (t) + V'' (\tilu_{j-1} (t) ) \tilh_{j-1} (t)
\end{align}
with initial data $\tilh_j (0) = 0$. By Lemma \ref{lem:ode-1} we have $\tilh_j (t) \leq 0$ for all $t$. For the $\tilk_j$ we have that they obey,
\begin{align}
\del_t \tilk_1 (t) &= - V'' (\tilu_1 (t) ) \tilk_1 (t)  \notag\\
\del_t \tilk_j (t) &= - V'' ( \tilu_j (t) ) \tilk_j (t) + V'' (\tilu_{j-1} (t) ) \tilk_{j-1} (t)
\end{align}
with initial data $\tilk_j (0) \leq 0$. By Lemma \ref{lem:ode-2} it follows that $\tilk_j (t) \leq 0$ for all $t$.  Then, with $F_j (t) := \del^2 \tilW_{N, t} $ for $\del^2 = \del_\eta \del_{\theta_2}$ or $\del^2 = \del_{\theta_2}^2$ we have that these obey the equations,
\beq
\del_t F_j (t) = - V'' (\tilu_j (t) ) F_j (t) + V'' ( \tilu_{j} (t) ) F_{j-1} (t) - V''' ( \tilu_j (t) ) g_j (t)
\eeq
where $g_j (t) = \tilh_j (t)^2$ or $g_j (t) = \tilh_j (t) \tilk_j (t)$, with initial data $F_j (0) =0$ and the convention that $F_0 (t) = 0$. In either case, $g_j (t) \geq 0$ and so since $V''' (x) \leq 0$ for all $x$, the nonnegativity of $F_j (t)$ for all $t$ and $j$ follows from Lemma \ref{lem:ode-4}. 
\qed

The following relates the derivative of $\tilW_{N, t}$ to the pseudo-Gibbs measure.
\bel \label{lem:three-param-2} For any $\theta >0$ we have,
We have,
\beq
( \del_{\theta_2} \tilW_{N, t} ) ( \theta, \theta, \theta) = \Ett_{N, t} [ (s_1-YN^{2/3} )_+ ].
\eeq
\eel
\proof First, note that the functions $\del_{\theta_2} \tilu_j (t)$ equal the $\hf_j (t)$ as defined in \eqref{eqn:hf-def} with the choice $f(s) = \1_{ \{ s > Y N^{2/3} \} }$. Indeed, by \eqref{eqn:tilh-system}, they are solutions to the same system of ODEs with the same initial data.  Therefore, 
\begin{align}
\del_{\theta_2} \tilW_{N, t} &= \sum_{j=1}^N \hf_j (t) + \int_0^t \1_{ \{ s_0 > Y N^{2/3} \} } \d s \notag\\
&= \Ett_{N, t} \left[ \int_0^{s_0} f(u) \d u \right] = \Ett_{N, t} [ (s_0- YN^{2/3} )_+ ] 
\end{align}
where in the second equality we used \eqref{eqn:psg-eq-1}. \qed

\subsection{Proof of lower bound with vanishing characteristic direction}

We first prove the lower bound in the special case that the quantity $t- N\psiV_1 ( \theta ) =0$ vanishes. The general case will be seen to follow easily from a perturbation argument and the general lower bound of Corollary \ref{cor:bad-lower}. 

Parts of the following proof draw some inspiration from the proof of the lower bound of \cite{MFSV}.  Our proof involves lower bounding $W_{N, t} ( \theta_0, \theta_0)$ by $W_{N, t} ( \lambda, \lambda)$ for some $\lambda $ close to $\theta_0$.  However, we have written the proof in the ``backwards'' direction, starting with a lower bound for $W_{N, t} ( \lambda, \lambda)$  and upper bounding this quantity by $W_{N, t} ( \theta_0 , \theta_0)$.  While this has the side effect of making some steps appear unmotivated (they are more easily motivated if one reads the proof in the other ``forwards'' direction), it is easier to verify the logic involved due to upper bounds being somehow conceptually simpler in this setting than lower bounds.

\bep \label{prop:lower}
Let $\theta_0 >0$ be such that $\psiV_2 ( \theta_0 ) < 0$. Define
\beq
t_0 := N \psiV_1 ( \theta_0).
\eeq
There is a constant $c_1 >0$ so that,
\beq
\Var \left(  W_{N, t_0} ( \theta_0 , \theta_0 ) \right) \geq c_1 N^{2/3}
\eeq
for all sufficiently large $N$.
\eep
\proof Let us denote,
\beq
Q := \ee[ W_{N, t_0} ( \theta_0 , \theta_0 ) ]
\eeq
and define $\lambda$ by
\beq
\lambda := \theta_0 + N^{-1/3}.
\eeq
Choose a $c_* >0$ so that,
\beq
Q + c_*^{-1} N^{1/3} \geq \ee[ W_{N, t_0 } ( \lambda , \lambda ) ] \geq Q + 4 c_* N^{1/3} ,
\eeq
for all sufficiently large $N$. 
Here we use the fact that the function $\theta \to \ee[ W_{N, t_0} ( \theta , \theta ) ] = \theta t_0 - N \psiV_0 ( \theta)$ has a derivative that vanishes at $\theta = \theta_0$ and a strictly positive second derivative in a small neighborhood of $\theta_0$ by the assumption that $\psiV_2 ( \theta_0) < 0$. 

Defining $ Z:= W_{N, t_0} ( \lambda, \lambda ) - Q$ we have by the Paley-Zygmund inequality,
\begin{align}
\pp[ W_{N, t_0} ( \lambda , \lambda) \geq Q + 2 c_* N^{1/3} ] &\geq \pp \left[ Z \geq \frac{1}{2} \ee[ Z] \right] \notag\\
&\geq \frac{1}{4} \frac{ \ee[ Z]^2}{ \Var(Z) + \ee[ Z]^2} \notag\\
&\geq 2 c_1
\end{align}
for some $c_1 >0$ and all $N$ sufficiently large. In the last line we used that $c_* N^{1/3} \leq \ee[Z] \leq CN^{1/3}$ for some $C>0$ and that $\Var(Z) \leq C N^{2/3}$ by Theorem \ref{thm: upper-bound}. From \eqref{eqn:1der-sign-2} we have that,
\beq
W_{N, t_0} ( \lambda , \lambda ) \leq W_{N, t_0} ( \theta_0 , \lambda)
\eeq
and so,
\beq
\pp[ W_{N, t_0} ( \theta_0 , \lambda ) \geq Q + 2 c_* N^{1/3} ] \geq \pp[ W_{N, t_0} ( \lambda , \lambda ) \geq Q + 2 c_* N^{1/3} ] \geq 2 c_1 ,
\eeq
for all sufficiently large $N$. 
Consider now the three parameter height function $\tilW_{N, t}$ as defined in the previous section.  We have,
\begin{align}
W_{N, t_0} ( \theta_0 , \lambda )  &= \tilW_{N, t_0} ( \theta_0 , \lambda , \lambda ) \notag\\
&=  \tilW_{N, t_0} ( \theta_0 , \lambda , \theta_0 ) + \int_{\theta_0}^\lambda ( \del_{\theta_2}  \tilW_{N, t_0} ) ( \theta_0 , \lambda , \theta_2 ) \d \theta_2 .
\end{align}
Now, by Lemmas \ref{lem:three-param-1} and \ref{lem:three-param-2}, respectively, we have
\begin{align}
 \int_{\theta_0}^\lambda  ( \del_{\theta_2} \tilW_{N, t_0} ) ( \theta_0 , \lambda , \theta_2 ) \d \theta_2 \leq&  \int_{\theta_0}^\lambda   ( \del_{\theta_2} \tilW_{N, t_0} )  ( \lambda , \lambda , \lambda) \d \theta_2 \notag\\
 = & N^{-1/3}  ( \del_{\theta_2} \tilW_{N, t_0} )  ( \lambda , \lambda , \lambda) \notag\\
 = & N^{-1/3} E^{ ( \lambda, \lambda )}_{N, t_0} [ ( s_0- Y N^{2/3} )_+ ].
\end{align}
By Corollary \ref{cor:conc} and Markov's inequality, we have that for $Y$ sufficiently large and all $N$ sufficiently large, 
\beq
\pp\left[ N^{-1/3} E^{ ( \lambda, \lambda )}_{N, t_0} [ ( s_0- Y N^{2/3} )_+ ] > c_* N^{1/3} \right] \leq c_1.
\eeq
Therefore, for $N$ sufficiently large we have,
\beq
\pp[ W_{N, t_0} ( \theta_0 , \lambda ) \geq Q + 2 c_* N^{1/3} ]  \leq \pp[ \tilW_{N, t_0} ( \theta_0 , \lambda, \theta_0 ) \geq Q +  c_* N^{1/3} ] + c_1
\eeq
and so,
\beq
\pp[ \tilW_{N, t_0} ( \theta_0 , \lambda, \theta_0 ) \geq Q +  c_* N^{1/3} ] \geq c_1.
\eeq
Now, by the Cameron-Martin theorem and Cauchy-Schwarz,
\begin{align}
&\pp[ \tilW_{N, t_0} ( \theta_0 , \lambda, \theta_0 ) \geq Q +  c_* N^{1/3} ]^2 \notag\\
= &\ee \left[ \1_{ \{ \tilW_{N, t_0} ( \theta_0 , \theta_0, \theta_0 ) \geq Q +  c_* N^{1/3} \} } \e^{( \theta_0 - \lambda) B_0 (Y N^{2/3} ) - Y( \lambda - \theta_0)^2 N^{2/3} /2 } \right]^2 \notag\\
\leq & \pp[ W_{N, t_0} ( \theta_0 , \theta_0 ) \geq Q +  c_* N^{1/3} ] \e^{-Y} \ee[ \e^{2( \theta_0 - \lambda) B_0 (Y N^{2/3} ) } ] \notag\\
= &\pp[ W_{N, t_0} ( \theta_0 , \theta_0 ) \geq Q +  c_* N^{1/3} ] \e^{Y}
\end{align}
We therefore conclude that there is some $c_2 >0$ so that for all $N$ sufficient large,
\beq
\pp[ W_{N, t_0} ( \theta_0 , \theta_0 ) - \ee[ W_{N, t_0} ( \theta_0 , \theta_0 ) ] \geq c_2 N^{1/3} ] \geq c_2.
\eeq
The claim now follows.
\qed

\subsection{Lower bound, general case}
\label{sec:lower-proof}

We first show that the variance is Lipschitz in the time parameter.
\bel \label{lem:var-comp-2}
For any $s, t >0$ and $\theta >0$ we have,
\beq
\left| \Var \left( W_{N, t} ( \theta , \theta  ) \right) - \Var \left( W_{N, s} ( \theta , \theta  ) \right) \right| \leq |t-s|
\eeq
\eel
\proof We have,
\beq
 \del_t  \left( \del_\theta W_{N, t} \right) ( \theta , \theta ) = - V'' (u_N ( \theta , \theta ) ) h_N ( t , \theta , \theta) \geq 0
\eeq
where the last inequality follows from \eqref{eqn:1der-sign-1} as well as
\beq
\del_t  \left( \del_\eta W_{N, t} \right) ( \theta , \theta ) = - V'' (u_N ( \theta , \theta ) ) k_N ( t , \theta , \theta) \geq 0
\eeq
where we again use \eqref{eqn:1der-sign-1}.

Assume now $s >t$. Then using the first line of \eqref{eqn:var-rep} we have,
\beq
\Var ( W_{N, t} ( \theta , \theta  ) ) \leq N \psiV_1 ( \theta) -t + 2 \ee[ ( \del_\theta W_{N, s} ) ( \theta , \theta ) ] = \Var( W_{N, s} ( \theta , \theta  ) ) +s - t .
\eeq
Similarly, the second line of \eqref{eqn:var-rep} gives,
\beq
\Var ( W_{N, t} ( \theta , \theta  ) ) \geq t - N \psiV_1 ( \theta) - 2 \ee[ ( \del_\eta W_{N, s} ) ( \theta , \theta ) ] = \Var( W_{N, s} ( \theta , \theta  ) )+t -s
\eeq
and so we conclude. \qed

\subsubsection{Proof of Theorem \ref{thm:lower}}

The entirety of this subsection is devoted to the proof of Theorem \ref{thm:lower}. Fix $\theta_0 >0$ satisfying the hypotheses of the theorem. Let $t_0 = N \psiV_1 ( \theta_0)$. By Lemma \ref{lem:var-comp-2} and Proposition \ref{prop:lower} we have that there is a $c >0$ so that for all $t$ satisfying,
\beq
|t - t_0 | \leq c N^{2/3}
\eeq
we have,
\beq
\Var ( W_{N, t} ( \theta_0 , \theta_0 ) ) \geq c N^{2/3}.
\eeq
On the other hand, if $|t - t_0 | = |t - N \psiV_1 ( \theta_0 ) | \geq c N^{2/3}$  then by Corollary \ref{cor:bad-lower} we have,
\beq
\Var ( W_{N, t} ( \theta_0 , \theta_0 ) )  \geq |t - N \psiV_1 ( \theta_0 ) | \geq c N^{2/3} .
\eeq
We conclude the proof. \qed

\section{Gaussian case}
\label{sec:gauss}

In this short section we provide the proof of Proposition \ref{prop: normal} and analyze the case $V(x) = \frac{x^2}{2}$. In this case, the system \eqref{eqn: system} reads 
\begin{align}
    \mathrm{d}u_1&=(-\theta -u_1)\mathrm{d}t +\mathrm{d}B_0+\mathrm{d}B_1 \notag\\
    \mathrm{d}u_j&=(u_{j-1}-u_j)\mathrm{d}t +\mathrm{d}B_j-\mathrm{d}B_{j-1}
\end{align}
These equations have explicit solutions
\begin{align}
    u_1(t)&=\int_0^t \e^{-\int_s^t \mathrm{d}r}(-\theta \mathrm{d}s +\mathrm{d}B_0(s)+\mathrm{d}B_1(s))+\e^{-t}u_1(0)\\
    &=\theta(\e^{-t}-1) +\int_0^t \e^{-(t-s)}\mathrm{d}B_0(s)+\int_0^t \e^{-(t-s)}\mathrm{d}B_1(s)+\e^{-t}u_1(0),\\
    u_j(t)&=\int_0^t \e^{-\int_s^t \mathrm{d}r}(u_{j-1}(s)\mathrm{d}s +\mathrm{d}B_j(s)-\mathrm{d}B_{j-1}(s))+\e^{-t}u_j(0)\\
    &=\int_0^t \e^{-(t-s)}u_{j-1}(s)\mathrm{d}s+\int_0^t \e^{-(t-s)}\mathrm{d}B_{j}(s)-\int_0^t 
    \e^{-(t-s)}\mathrm{d}B_{j-1}(s)+\e^{-t}u_j(0) \text{ for } j\geq 2.
\end{align}
In particular, by induction we see that each $u_j (t)$ is a linear combination of $ \{ u_i (0) \}_{1 \leq i \leq j }$ and Wiener integrals against the $B_i (t)$. In particular, in this case the height function $W_{N, t}^\theta$ is a linear combination of jointly Gaussian random variables and so is Gaussian, and its distribution is therefore completely determined by its mean and variance. 

Consider now the functions,
\beq
h_n(t) :=-\int_0^t \frac{s^{n-1}}{(n-1)!}\e^{-s}ds
\eeq
and
\beq
f_n (t) := -\ee[ u_n (t) B_0 (t)]
\eeq
Clearly,
\beq
h_1(t) = f_1 (t) = \e^{-t}-1.
\eeq
On the other hand we have,
\beq
f_n (t) = \int_0^t \e^{ - (t-s)} f_{n-1} (s) \d s .
\eeq
Now by integration by parts,
\beq
h_n (t) = \frac{ t^{n-1} \e^{ - t }}{ (n-1)!} - h_{n-1} (t) = - \del_t h_n (t) - h_{n-1} (t).
\eeq
Therefore,
\beq
h_n(t) = \int_0^t \e^{ - (t-s)} h_{n-1} (s) \d s
\eeq
and so $h_n(t) = f_n (t)$ for all $n$ and $t$. We have,
\beq
\Var (W_{N, t}^\theta ) = N - t - 2 \ee[ W_{N, t}^\theta B_0 (t) ]
\eeq
Now, since $\d W_{N, t}^\theta = -u_N (t) \d t + \d B_N$ we have,
\begin{align*}
    - \ee[ W_{N, t}^\theta B_0 (t) ]&=-\int_0^t h_N(s)\mathrm{d}s\\
    &=\int_0^t\int_0^s \frac{r^{N-1}}{(N-1)!}\e^{-r}\mathrm{d}r\mathrm{d}s\\
    &=\int_0^t\int_r^t \frac{r^{N-1}}{(N-1)!}\e^{-r}\mathrm{d}s\mathrm{d}r\\
    &=\int_0^t (t-r)\frac{r^{N-1}}{(N-1)!}\e^{-r}\mathrm{d}r\\
    &=-t h_N(t)+N h_{N+1}(t)\\
    &=(N-t)h_N(t)+N\left[h_{N+1}(t)-h_N(t)\right]\\
    &=(N-t)h_N(t)+\frac{t^N}{(N-1)!}\e^{-t}.
\end{align*}
Therefore, 
\begin{align*}
    \mathrm{Var}[W_{N,t}]&=N-t+2\mathbb{E}[\partial_\theta W_{N,t}]\\
    &=(N-t)\left[1+2h_N(t)\right]+2\frac{t^N}{(N-1)!}\e^{-t}.
    \end{align*}
    When $t = N$ we use Stirling's approximation to get
        \beq
        \frac{\Var (W_{N, N}^\theta)}{\sqrt{N}}=\frac{2 N^N \e^{-N}}{(N-1)! \sqrt{N}}=2\left(\frac{N}{e}\right)^N\frac{\sqrt{N}}{N!} = \sqrt{\frac{2}{\pi}} + o (1)
        \eeq
        This yields the claim. \qed

\appendix

\section{Well-posedness; generator and invariant measure} \label{a:diff-1}

The purpose of this appendix is to prove Proposition \ref{prop:inv-measure}. The well-posedness  component of the proposition statement follows from Proposition \ref{prop:a-wp}. The characterization of the invariant measure is the content of Appendix \ref{a:inv-measure}. 

\subsection{Well-posedness}

In this section we deal with well-posedness of the system \eqref{eqn: system}. Since the coefficients of the Brownian terms are constant, the system \eqref{eqn: system} may be re-interpreted as a system of ordinary differential equations, for which classical results allow one to obtain a solution. The only thing that must be checked is that the system does not \emph{explode} in finite time, that is, the solution remains bounded on bounded time intervals. Here, the main point is that due to the confining nature of the potential and triangular nature of the system, it is straightforward to check the non-explosion. The following lemma will be used iteratively in the just-described proof.

\bel \label{lem:simple-wp}
Let $W : \rr \to \rr$ be a continuous function satisfying,
\beq
W(x) \mathrm{sign}(x) \geq - L (|x| + 1)
\eeq
for some $L >0$. Let $f (s) : \rr_+ \to \rr$ be a continuous function and let $u(t)$ be a continuous function satisfying,
\beq
u(t) = - \int_0^t W(u(s) ) \d s + f(t)
\eeq
on some time interval $[0, T]$. If $M > 1$ satisfies
\beq
M \geq \sup_{s \in [0, T] } |f(s)| + |u(0) |
\eeq
then,
\beq
\sup_{ t \in [0, T] } |u (t)| \leq 1 + (4M+LT) \left( 1 + TL \e^{ LT } \right) =: C_T.
\eeq
\eel
\proof Suppose for a contradiction that there exists a $t < T$ such that $|u(t)| = C_T$. Let
\beq
t_* := \inf \{ t \in [0, T] : |u(t) | = C_T \}.
\eeq
We assume that $u(t_* ) = C_T > 0$. The case when $u(t_*) = - C_T < 0$ is similar. Let,
\beq
s_* = \sup \{ t \in [0, t_*] : |u(t) | = 2M\}.
\eeq
Then since $u(t)$ is continuous we have $0 < s_* < t_* \leq T$, and moreover, for $s_* \leq t \leq t_*$,
\beq
u (s_* ) = 2M \leq u(t) \leq C_T = u (t_*).
\eeq
It follows that for $ s\in [s_*, t_*]$ that $W(u(s) ) \geq -L(u(s) +1)$ and so for $t \in (s_*, t_*]$ we have
\beq
u(t)  = u(s_*) + f(t) - f(s) - \int_{s_*}^t W(u(s) ) \d s \leq ( 4 M + LT ) + \int_{s_*}^t L u(s) \d s
\eeq
By Gronwall's integral inequality,
\beq
u(t) \leq (4M+LT) \left( 1 + TL \e^{ LT } \right), 
\eeq
which yields a contradiction upon taking $t = t_*$.  \qed

\bep \label{prop:a-wp} Let $V$ be of O'Connell-Yor type. 
For each choice of initial data and each realization of the Brownian motions, there exists a unique global-in-time solution of \eqref{eqn: system}. Moreover, for each fixed realization of the Brownian motion, the solutions of \eqref{eqn: system} are uniformly bounded as $t, \theta$ and the initial data vary over compact subsets of $[0, \infty) \times (0, \infty) \times \rr^N$.  Consequently, the system \eqref{eqn: system} defines a Markov process taking values in $C ( [0, \infty) , \rr^N)$. 
\eep
\proof The system \eqref{eqn: system} has the form of an $N$-dimensional stochastic differential equation:
\begin{equation}
    \mathrm{d}x= b(x)\mathrm{d}t+\sigma \mathrm{d}B,
\end{equation}
with smooth coefficients
\begin{equation}\label{eqn: b-coeff}
b(x)=(-\theta-V'(x_1),V'(x_1) -V'(x_2),\ldots, V'(x_{N-1})-V'(x_N))^T\in \mathbb{R}^N,
\end{equation}
and $\sigma \in M_{N \times (N+1)}(\mathbb{R})$ is given by
\begin{equation}\label{eqn: sigma}
\sigma=\left(\begin{array}{cccccc}
1 & 1 &  &  &  &\\
 & -1 & 1 &  & & \\
&    & -1 & 1  & &\\
 & & & \ddots &  &\\
 &  &  & -1 & 1 & \\
 &  &  &  & -1 & 1
\end{array}\right),
\end{equation}
and the Brownian motion $B$ is $B=(B_0,B_1,\ldots,B_N)^T$.

We start by considering, for a fixed continuous $f:\mathbb{R}_+\rightarrow \mathbb{R}^N$, the equations
\begin{equation}\label{eqn: vf-equation}
x(t)=x(0)+\int_0^t b(x(s))\mathrm{d}s+ f(t).
\end{equation}
This is a system of integral equations with right-hand side $F(x,t)=\int_0^t b(x(s))\,\mathrm{d}s+f(t)$, where $F$ locally Lipschitz in $x$ (the Lipschitz constant does not depend on $f$). By classical results (see for example \cite[Theorem 1, Chapter 21]{varadhan}) there exists, for each initial data $x(0)=x_0$ and each continuous $f$, a maximal time $\tau=\tau_\infty(f)>0$ of existence and a unique solution $x(t)=x(t,u_0,f)\in\mathbb{R}^N$ of \eqref{eqn: vf-equation} on $(0,\tau)$. Moreover, if $\tau_\infty(f)<\infty$, then $\lim_{t\rightarrow \tau_\infty^-}|x(t)|=\infty$. $\tau_\infty(f)$ is called the \emph{explosion time}.

We now claim that the explosion time satisfies $\tau_\infty(f)= \infty$, for every choice of $f$.  First, note that by assumption of \eqref{eqn: assumption} we have $V'(x) \leq 0$ for all $x$.  By this and the convexity of $V(x)$ we see that $W(x) = V'(x)$ satisfies the assumptions of Lemma \ref{lem:simple-wp}. Applying this lemma we see that $x_1(t)$ does not explode in finite time. For higher $j$, we write the system \eqref{eqn: system} as
\beq
x_j (t) = -\int_0^t V' (x_j (s) ) \d s + F(t),
\eeq
where $F(t)$ depends on $x_{j-1}$ and the Brownian motion terms.
Arguing inductively we see that Lemma \ref{lem:simple-wp} implies that if $x_{j-1}$ does not explode in finite time, then neither does $x_j$.  Evaluating the solution of \eqref{eqn: vf-equation} at each realization of the Brownian motion sample paths yields the solution to the system \eqref{eqn: system}. This shows that the system has global-in-time solutions. The claim about uniform boundedness follows from tracking the constants in the iterative applications of Lemma \ref{lem:simple-wp}.  The fact that this defines a Markov process is clear, as the solution to \eqref{eqn: vf-equation} depends only on the initial data and the increments of the Brownian motion sample paths.   \qed
   
   \subsection{Generator and invariant measure} \label{a:inv-measure}
   
   Denote $\del_j = \frac{ \del}{ \del x_j}$.  
We begin by remarking that system \eqref{eqn: system} is an It\^o diffusion with generator
\begin{equation}
\begin{split}
L&= \frac{1}{2}(\partial_1)^2+\frac{1}{2}\sum_{j=2}^N\left(\partial_j-\partial_{j-1}\right)^2+\frac{1}{2}\left(\partial_N\right)^2 \label{eqn:elliptic}\\
&+ \left(-V'(u_1)-\theta\right)\partial_1 + \sum_{j=2}^N \left[V'(u_{j-1})-V'(u_j)\right]\partial_j.
\end{split}
\end{equation}
This generator $L$ has the form
\beq
L= \frac{1}{2} \nabla^T a \nabla  +b^T \nabla
\eeq
where $a\in M_{N\times N}(\mathbb{R})$ is given by,
\beq
a :=\left(\begin{array}{cccccc}
2 & -1 &  &  &  &\\
-1 & 2 & -1 &  & & \\
&   -1 & 2 & -1 & &\\
 & & & \ddots &  &\\
 &  &  & -1 & 2 & -1\\
 &  &  &  & -1 & 2
\end{array}\right) = S + S^T ,
\eeq
where the difference operator has matrix elements $S_{ij} = \delta_{ij} - \delta_{i-1,j}$,  and the function $b$ is as in \eqref{eqn: b-coeff}. Note that if we define $U$ by
\beq
U= \sum_{j=1}^N V(u_j) + \theta u_j .
\eeq
then $b = - S \nabla U$.  From this we immediately see that the function $\e^{-U}$ satisfies the equation $L^* \e^{ - U} = 0$. Indeed, we have
\beq
\frac{1}{2} \nabla^T a \nabla \e^{ - U} = \nabla^T S \nabla \e^{ - U} = -\nabla^T (\e^{ - U} S \nabla U ) = \e^{ - U} \left( ( \nabla U)^T S \nabla U - \nabla^T S \nabla U \right)
\eeq
as well as,
\beq
\nabla^T \left( S ( \nabla U ) \e^{ - U} \right) = \e^{ - U} \left( \nabla^T S \nabla U - ( \nabla U )^T S \nabla U \right) ,
\eeq
from which it follows $L^* \e^{-U} = 0$. 
In order to show that $\e^{ - U}$ indeed defines an invariant measure we make a change of variable bringing $L$ into the form of a perturbation of the Laplacian. For this, we first establish the following.
\begin{prop}The symmetric matrix $a$ is positive definite. \end{prop}

\begin{proof}
First, we establish that the eigenvalues of the matrix $a$ are nonnegative. Indeed, we have,
\begin{align}
\frac{1}{2}\xi^t a \xi &= \frac{1}{2} \e^{ - x \cdot \xi } \nabla^T a \nabla \e^{ x \cdot \xi } \notag\\
&=\e^{ - x \cdot \xi }\left\{\frac{1}{2}(\partial_1)^2+\frac{1}{2}\sum_{j=2}^N\left(\partial_j-\partial_{j-1}\right)^2+\frac{1}{2}\left(\partial_N\right)^2\right\} \e^{x\cdot \xi}\geq 0,
\end{align}
where the second line follows from the fact that $\frac{1}{2} \nabla^T a \nabla$ gives the second-order part of the generator $L$ in \eqref{eqn:elliptic}. 
Moreover, denoting by $D_n$ the determinant of the matrix $a$ in dimension $n$, we have for $n\geq 2$
\[D_n = 2 D_{n-1}-D_{n-2}, \quad D_1=2, \quad D_2=3.\]
from which one easily verifies $D_n=n+1>0$. \qed
\end{proof}
Now, let $A = a^{-1/2}$ be the unique positive definite symmetric square root of $a^{-1}$ and consider the coordinates $v = A u$.  Since,
\beq
c \| x \|_2 \leq \| A x \|_2 \leq C \| x \|_2
\eeq
for some $c, C>0$ we see that process defined by $v(t) = A u(t)$ explodes if and only if $u$ does. We conclude that $v$ is also a Markov process with no explosion and moreover the invariant measures of $v$ are in one-to-one correspondence with those of $u$.  By the change of variable,
\beq
\nabla_u = A \nabla_v
\eeq
we see that the generator $\tilde{L}$ of $v$ given by 
\beq
\tilde{L} = \frac{1}{2} (A \nabla_v )^T a A \nabla_v + b(A^{-1} v)^T A \nabla_v = \frac{1}{2} \Delta_v + \tilde{b} (v)^T \nabla_v
\eeq
where we defined $\tilde{b} (v) := A b ( A^{-1} v )$.  We will obtain the uniqueness of the invariant measure from the following result \cite[Ch. 31, p. 254]{varadhan}, which is formulated for perturbations of the Laplacian such as $\tilde{L}$.
\begin{theorem}[]\label{thm:varadhan-invariant}
Let $D=\frac{1}{2}\Delta+B\cdot \nabla$ for which the corresponding diffusion does not \emph{explode}, that is, it almost surely remains bounded on bounded time intervals. (See \cite[Chapter 24]{varadhan} for more information, including criteria for non-explosion.) Assume $B(x)$ is $C^\infty(\mathbb{R}^N;\mathbb{R}^N)$. Define the formal adjoint $D^*$ of $D$ by
\[D^*=\frac{1}{2}\Delta-\nabla\cdot B.\]
Suppose there exists a smooth function $\varphi$ such that $D^* \varphi=0$, $\varphi \geq 0$. Then $\mu(A)=\int_A \varphi(y)\,\mathrm{d}y$ defines a unique invariant distribution for the process.
\end{theorem}
If $\Phi_A$ denotes the (unitary) composition map $( \Phi_A  f ) (x) := \det(A)^{-1/2} f (Ax)$ then, noting that $\tilde{L} = \Phi_A^* L \Phi_A$, and so $\tilde{L}^* = \Phi_A^* L^* \Phi_A$ we see that $\Phi_A^{-1} ( \e^{-U})$ defines a unique invariant measure for $\tilde{L}$. From the above discussion we see that the measure $\omega_\theta$ is the unique invariant measure for the Markov process with generator $L$.

\section{Derivatives} \label{a:diff}

In this section we deal with proving differentiability of the solutions $u_j (t , \eta , \theta)$ in the parameters $\eta$ and $\theta$. In general, the finite difference quotients and the derivatives satisfy various systems of ODEs. Therefore, we begin with a short section containing a few different kinds of systems of ODEs that we encounter and state some of their positivity-preserving properties, as well as standard contractivity properties, etc. 

\subsection{ODE lemmas} 
\label{a:ode}
\bel \label{lem:ode-1}
Let $W_j : \rr_+ \to \rr$ be nonnegative continuous functions. Let $f : \rr_+ \to \rr$ be a nonnegative piecewise continuous function. Then, the solution of the inhomogeneous linear system of ODEs,
\begin{align}
\del_t v_1 (t) &= - W_1 (t) v_1 (t) - f(t) \notag\\
\del_t v_j (t) &= - W_j (t) v_j (t) + W_{j-1} (t) v_{j-1} (t) , \qquad j \geq 2
\end{align}
with $v_j (0) = 0$ for all $j$ satisfies,
\beq
0 \leq - v_j (t) \leq \int_0^t f(s) \d s =: F(t)
\eeq
for all $t$. We also have,
\beq \label{eqn:ode-1-integrals}
0 \geq \int_0^t W_j (s) v_j (s) \d s \geq \int_0^t W_{j-1} (s) v_{j-1} (s) \d s  \geq \dots \geq - F(t).
\eeq
\eel 
\proof Recall the solution of $\del_t u = - W u -g$ is given by
\beq \label{eqn:1d-ode}
u(t) = \exp\left( -\int_0^t W(s) \d s \right) u(0) - \int_0^t \exp \left(- \int_s^t W(u) \d u \right) g(s) \d s.
\eeq
Applying this we first see that $v_1 (t ) \leq 0$ for all $t$ and then that $v_j (t) \leq 0$ for all $t$ if $v_{j-1} (t) \leq 0$ for all $t$. Hence, we see that the $v_j$'s are all nonpositive.  The bound $v_1 (t) \geq - F(t)$ follows from integrating
\beq
\del_t v_1 (t) =  - W_1 (t) v_1 (t) - f(t) \geq - f(t).
\eeq
On the other hand we also obtain,
\beq
\int_0^t W_1 (s) v_1 (s) \d s =  - F(t) - v_1 (t) \geq - F(t).
\eeq
Integrating the equation for $\del_t v_j$ gives,
\beq
\int_0^t W_j(s) v_j (s) \d s= \int_0^t W_{j-1} (s) v_{j-1} (s) \d s - v_j (t) \geq \int_0^t W_{j-1} (s) v_{j-1} (s) \d s,
\eeq
and so  \eqref{eqn:ode-1-integrals} follows by induction.  
Similarly, by integrating the equations for $\del_t v_j$
\beq
v_j (t) \geq \int_0^t W_{j-1} (s) v_{j-1} (s) \geq - F(t).
\eeq
This finishes the proofs. \qed

\bel \label{lem:ode-4} Let $W_n (t)$ be continuous nonnegative functions and $g_n (t)$ be nonnegative piecewise continuous. The solution $w_n(t)$ to the system
\begin{align}
\del_t w_1 (t) &= - W_1 (t) w_1 (t) + g_1 (t) \notag\\
\del_t w_j (t) &= - W_j (t) w_j (t) + W_{j} (t) w_{j-1} (t) + g_j (t), \qquad j \geq 2
\end{align}
with initial data $w_n (0) = 0$ for all $n$ satisfies $w_n (t) \geq 0$ for all $n$ and $t$.
\eel
\proof  This follows immediately from \eqref{eqn:1d-ode} and induction. \qed

\bel \label{lem:ode-2}
Let $W_j : \rr_+ \to \rr$ be  continuous nonnegative functions. Consider the homogeneous linear ODE,
\begin{align} \label{eqn:homog-ode}
\del_t v_1(t) &= - W_1 (t) v_1 (t) \notag\\
\del_t v_j (t) &= - W_j (t) v_j (t) + W_{j-1} (t) v_{j-1} (t) , \qquad j \geq 2
\end{align}
If the initial data are nonnegative then the $v_j(t)$ are nonnegative for all times $t$. Moreover,
\beq \label{eqn:l1-dec-1}
0 \leq \sum_{j=1}^n v_j (t) \leq \sum_{j=1}^n v_j (0),
\eeq
and
\beq
\int_0^t W_n (s) v_n (s) \d s \leq \sum^n_{j=1} v_j (0)
\eeq
For arbitrary initial data we have,
\beq \label{eqn:l1-dec-2}
\left| \sum_{j=1}^n v_j (t) \right| \leq \sum_{j=1}^n |v_j (0) |.
\eeq
\eel
\proof First consider the case of nonnegative initial data. From the explicit form \eqref{eqn:1d-ode}  of the solution we conclude the non-negativity of the $v_j (t)$ for all times $t$. We have,
\beq
\del_t \left( \sum_{j=1}^n v_j (t) \right) = - W(t) v_j (t) \leq 0
\eeq
and so \eqref{eqn:l1-dec-1} follows.  Via direct integration we have
\beq
\int_0^t W_n (s) v_n (s) \d s = v_n (0) - v_n (t) + \int_0^t W_{n-1} (s) v_{n-1} (s) \d s \leq v_n (0)  + \int_0^t W_{n-1} (s) v_{n-1} (s) \d s 
\eeq
and so the other estimate follows by induction.

 For the second claim, consider the solution $g_j (t)$ of \eqref{eqn:homog-ode} with initial data $g_j (0) = |v_j (0)|$. Then $w_j (t) := g_j (t) - v_j (t)$ also solves \eqref{eqn:l1-dec-1} and is nonnegative for all times $t$, as the initial data is nonnegative. Hence, 
\beq
\sum_{j=1}^n v_j (t) \leq \sum_{j=1}^n g_j (t) \leq \sum_{j=1}^n g_j (0) = \sum_{j=1}^n |v_j (0) |
\eeq
where in the second inequality we applied \eqref{eqn:l1-dec-1} to $g_j(t)$. 
The upper bound of \eqref{eqn:l1-dec-2} follows. The lower bound follows by considering $w_j (t) := g_j (t) + v_j (t)$. \qed

\bel \label{lem:ode-3}
Let $W_j, g_j : \rr_+ \to \rr$ be nonnegative continuous functions. Consider the system of ODEs,
\begin{align}
\del_t v_1 (t) &= - W_1 (t) v_1 (t) + g_1 (t) \notag\\
\del_t v_j (t) &= - W_j (t) v_j (t) + W_{j-1} (t) v_{j-1} (t) + g_j (t) - g_{j-1} (t) ,
\end{align}
with $v_j (0) = 0$.
Then for any $n$ we have for all $t$,
\beq
\sum_{j=1}^n v_j (t) \geq 0 ,
\eeq
and the estimates
\beq
  -\sum_{j=1}^{n-1} \int_0^t g_j (s) \leq v_n (t)  \leq \sum_{j=1}^n \int_0^t g_j (s) \d s.
\eeq
\eel
\proof The partial sums $e_k := \sum_{j \leq k } v_j$ satisfy,
\beq
\del_t e_k(t) = - W_k(t) (e_k (t) - e_{k-1} (t) ) +g_k (t) ,
\eeq
where $e_0 = 0$. From Lemma \ref{lem:ode-4} we conclude that $e_n (t) \geq 0$ for all $n$ and $t$. 

We now turn to the estimates. For any $k \geq 1 $ let $\{ m^k_j (t)\}_{j=1}^\infty$ satisfy,
\begin{align}
\del_t m^k_1 (t) &= - W_1 (t) m^k_1 (t) - \delta_{1k} g_1 (t) \notag\\
\del_t m^k_j (t) &= - W_j (t) m^k_j (t) + W_{j-1} m^k_{j-1} (t) - \delta_{jk} g_j (t) , \qquad j\geq 2
\end{align}
with initial condition $m^k_j (0) = 0$ 
and for $k \geq 2$ let $\{ w^k_j \}_{j=1}^\infty (t)$ satisfy,
\begin{align}
\del_t w^k_1 (t) &= - W_1 (t) w^k_1 (t)\notag\\
\del_t w^k_j (t) &= - W_j (t) w^k_j (t) + W_{j-1} w^k_{j-1} (t) - \delta_{jk} g_{j-1} (t) , \qquad j\geq 2
\end{align}
with initial condition $w^k_j (0) = 0$.  Note that $m_j^k$ and $w_j^k$ are identically $0$ for $j < k$ and moreover,
\beq
v_n (t) = \sum_{k=1}^n ( w_j^k (t) - m_j^k (t) ),
\eeq
by linearity of the equations, where we set $w^1_j =0$. By Lemma \ref{lem:ode-1} we have,
\beq
0 \leq -w_j^k (t) \leq \int_0^t g_{k-1} (s) \d s, \qquad 0 \leq - m_j^k (t) \leq \int_0^t g_k (s) \d s.
\eeq
The claim follows. \qed

\subsection{First derivatives} \label{a:1-der}

\bep \label{prop:1d-finite-dif-d}
Let $V$ be of O'Connell-Yor type. Fix $\theta, \eta >0$. For $|h| < \theta \wedge 1$ define the difference quotient,
\beq
\Deld_{j, h} (t, \eta, \theta ) := \frac{ u_j (t, \eta, \theta+h) - u_j ( t , \eta , \theta ) }{h}.
\eeq
Then, the estimate,
\beq
0 \leq - \Deld_{j, h} (t, \eta, \theta) \leq t
\eeq
holds. Additionally, define,
\beq
F_j (t) := \int_0^1 V'' ( \tau u_j ( t , \eta, \theta + h ) + ( 1 - \tau ) u_j (t, \eta, \theta ) ) \d \tau ,
\eeq
Then we have the estimate,
\beq
\int_0^t (- \Deld_{j, h} (s) ) F_j (s) \d s \leq t .
\eeq
\eep
\proof For notational simplicity, let us denote $\Deld_j (t) = \Deld_{j, h} (t, \eta, \theta)$.  By direct calculation the $\Deld_j$ satisfy,
\begin{align}
\del_t \Deld_1 (t) &= - F_1 (t) \Deld_1 (t) - 1 \notag\\
\del_t \Deld_j (t) &= - F_j (t) \Deld_j (t) + F_{j-1} (t) \Deld_j (t) , \qquad j \geq 2
\end{align}
where $F_j$ is as above. The initial data satisfies $\Deld_j (0) = 0$. The claim now follows from Lemma \ref{lem:ode-1}, since $V'' \geq 0$ by assumption. \qed

\bec \label{cor:1st-der-1}
Let $V$ be of O'Connell-Yor type. The functions $u_j (t, \eta, \theta)$ are differentiable in $\theta$ and the derivatives,
\beq
h_j (t) = h_j (t, \eta, \theta) := \del_\theta u_j (t, \theta, \eta)
\eeq
satisfy the system of ODEs,
\begin{align} \label{eqn:system-1d-d}
\del_t h_1 (t) &= - V'' (u_1 (t) ) h_1 (t) - 1 \notag\\
\del_t h_j (t) &= - V'' (u_j (t) ) h_j (t) + V'' (u_{j-1} (t) ) h_{j-1} (t) , \qquad j \geq 2
\end{align}
where $u_j (t) = u_j (t, \eta, \theta)$. Moreover, the $h_j$ are jointly continuous in $\theta$ and $\eta$ and satisfy,
\beq
0 \leq -h_j (t) \leq t.
\eeq
\eec
\proof Denote the solution to \eqref{eqn:system-1d-d} by $h_j (t)$ and the difference quotients of Proposition \ref{prop:1d-finite-dif-d} by $\Deld_j(t)$. By the calculations given there we have that the difference $w_j := h_j - \Deld_j$ satisfies,
\begin{align}
\del_t w_1 (t) &= - V'' (u_1(t) ) w_1 (t) + \Deld_1 (t) (V''(u_1(t) ) - F_1 (t) ) \notag\\
\del_t w_j (t) &= - V''(u_j(t) ) w_j (t) + V'' (u_{j-1} ) w_{j-1} (t) \notag\\
 & + \Deld_j (t) (V''(u_j(t) ) - F_j (t) ) - \Deld_{j-1} (t) (V''(u_{j-1}(t) ) - F_{j-1} (t) ) , \qquad j \geq 2 ,
\end{align}
where the $F_j$ are as in the proof of Proposition \ref{prop:1d-finite-dif-d}. The estimates of Proposition \ref{prop:1d-finite-dif-d} and the fact that $V$ is smooth imply that the inhomogeneous terms in the above system all tend to $0$ uniformly in any interval $[0, T]$. Therefore, by the explicit form of the solution of the above system, one sees that the $w_j$ tend to $0$ as $h \to 0$ uniformly in $t$.

The joint continuity follows from the fact that the coefficients in the ODEs satisfied by the $h_j$'s are continuous in $\theta$ and $\eta$.  \qed

Recall the initial data $u_j (0, \eta, \theta)$ are given by $H_\theta (q_j)$ where $H_\theta = F_\theta^{-1}$. Defining $G(\theta, x, y ) : \rr_+ \times (0, 1) \times \rr \to \rr$ by 
\beq
G ( \theta, x, y) := F_\theta (y) - x
\eeq
we see that by the implicit function theorem and the smoothness of $F$ that the function $H_\theta (x)$ which satisfies $G(\theta, x, H_\theta (x)) = 0$ is smooth in $\theta$ and $x$. 

Let $X$ denote a random variable distributed according to $\nu_\eta$. Then,
\beq
\del_\eta F_\eta (u) = - \Cov ( X , \1_{ \{ X \leq u \} } ) \geq 0,
\eeq
where we used the general fact that $\Cov (Y, F(Y) ) \geq 0$ for any random variable $Y$ and increasing $F$, provided the covariance exists. Moreover by differentiating $F_\eta ( H_\eta (q) ) = q$ we see that,
\beq \label{eqn:delH-sign}
\del_\eta H_\eta (q) = - \frac{ ( \del_\eta F_\eta ) ( H_\eta (q) ) }{ F'_\eta ( H_\eta (q) )}  \leq 0.
\eeq
\bep \label{prop:1d-finite-dif-i}  
Let $V$ be of O'Connell-Yor type. Fix $\theta, \eta >0$. For $|h| < \eta \wedge 1$ define the difference quotient,
\beq
\Deli_{j, h} (t, \eta, \theta) := \frac{ u_j ( t, \eta + h , \theta ) - u_j ( t , \eta , \theta ) }{ h}.
\eeq
Then, the $\Deli_{j, h} (t , \eta , \theta)$ are all non-positive and we have the estimate,
\beq
 - \sum_{j=1}^n \Deli_{j, h} ( t , \eta, \theta) \leq - \sum_{j=1}^n \Deli_{j, h} ( 0 , \eta, \theta) 
\eeq
 Additionally, define,
\beq
G_j (t) := \int_0^1 V'' ( \tau u_j ( t , \eta+h, \theta  ) + ( 1 - \tau ) u_j (t, \eta, \theta ) ) \d \tau .
\eeq
Then we have the estimate,
\beq
\int_0^t (- \Deli_{j, h} (t, \eta, \theta ) ) G_j (s) \d s \leq \sum_{k=1}^j \Deli_{k, h} (0, \eta, \theta)
\eeq
\eep
\proof For notational simplicity we denote $\Deli_j (t) = \Deli_{j, h} (t, \eta, \theta)$. By direct calculation they satisfy the system of ODEs,
\begin{align}
\del_t \Deli_1 (t) &= - G_1 (t) \Deli_1 (t) \notag \\
\del_t \Deli_j (t) &= - G_j (t) \Deli_j (t) + G_{j-1} (t) \Deli_{j-1} (t) , \qquad j \geq 2.
\end{align}
The claim follows by Lemma \ref{lem:ode-2} and the fact that \eqref{eqn:delH-sign} implies that $\Deli_j (0) \leq 0$ for all $j$. \qed

Given the above proposition, the proof of the following is almost identical to the proof of Corollary \ref{cor:1st-der-1} given Proposition \ref{prop:1d-finite-dif-d} and is omitted. Note that \eqref{eqn:delH-sign} gives that the $k_j (0)$ defined below satisfy $k_j (0) \leq 0$ for all $j$. 
\bec \label{cor:1st-der-2}
Let $V$ be of O'Connell-Yor type. The functions $u_j (t, \eta, \theta)$ are differentiable in $\eta$ and the derivatives
\beq
k_j (t) = k_j (t, \eta, \theta)  = \del_\eta u_j (t, \eta, \theta) 
\eeq
satisfy the system of ODEs,
\begin{align}
\del_t k_1 (t) &= - V'' (u_1 (t) ) k_1 (t) \notag\\
\del_t k_j (t) &= - V'' (u_j (t) ) k_j (t) + V'' (u_{j-1} (t) ) k_{j-1} (t) , \qquad j \geq 2
\end{align}
where $u_j (t) = u_j (t , \eta, \theta)$. Moreover, the $k_j$ are jointly continuous in $\theta$ and $\eta$ and satisfy the inequality,
\beq
0 \leq - k_n ( t , \eta , \theta ) \leq - \sum_{j=1}^n k_j (0, \eta , \theta) .
\eeq
\eec

\subsection{Second derivatives} \label{a:2-der}

\bep \label{prop:2-dd}
Consider the difference quotients,
\beq
\Deldd_{j, h} (t, \eta, \theta) := \frac{ h_j (t, \eta, \theta + h ) - h_j (t ,\eta, \theta ) }{h}
\eeq
Then,
\beq
| \Deldd_{j, h} (t , \eta , \theta ) | \leq C jt^2(1+t)
\eeq
for some $C>0$.
\eep
\proof For simplicity of notation let $\Deldd_j (t) = \Deldd_{j, h} (t, \eta, \theta)$. These satisfy the system of ODEs,
\begin{align}
\del_t \Deldd_1 (t) &= - V'' ( u_1 ( t , \eta, \theta ) ) \Deldd_1 (t) - Q_1 (t) \cdot h_1 (t , \eta, \theta+h) \notag\\
 \del_t \Deldd_j (t) &= -  V'' ( u_j ( t , \eta, \theta ) ) \Deldd_j (t) +  V'' ( u_{j-1} ( t , \eta, \theta ) ) \Deldd_{j-1} (t) \notag\\
 &+ Q_j (t) \cdot h_j (t, \eta , \theta + h ) - Q_{j-1} (t) \cdot h_{j-1}(t, \eta, \theta+h ) 
\end{align}
where,
\beq
Q_j (t) := \Deld_{j, h} ( t , \theta, \eta) \int_0^1 V''' ( \tau u_j (t, \eta, \theta + h ) + (1 - \tau ) u_j ( t , \eta , \theta ) ) \d \tau.
\eeq
By the assumption that $0 \leq -V''' (x)  \leq C( V''(x)+1)$ we see that,
\beq \label{eqn:2-dd-proof-1}
0 \leq   Q_j (t) \leq -C \Deld_{j, h} ( t , \eta , \theta) (F_j (t)+1)
\eeq
 where $F_j$ are from Proposition \ref{prop:1d-finite-dif-d}. The estimates now follow from Lemma \ref{lem:ode-3} and the estimates of the time integrals of the $F_j$ of Proposition \ref{prop:1d-finite-dif-d}, as well as the estimate for $h_j (t)$ given in Corollary \ref{cor:1st-der-1}.  \qed

\bep
Consider the difference quotients,
\beq
\Delii_{j, h} (t, \eta, \theta) := \frac{ k_j (t, \eta + h , \theta ) - k_j (t, \eta, \theta ) }{h}. 
\eeq
We have,
\beq
| \Delii_{n, h} (t , \eta , \theta) | \leq \sum_{j=1}^n | \Delii_{n, h} (0, \eta, \theta ) | + Cn (1+t) \left( \sum_{j=1}^n | \Deli_{j,h} ( 0, \eta, \theta) | \right) \left( \sum_{j=1}^n -k_j (0, \eta+h, \theta ) \right)
\eeq
\eep
\proof Denote $\Delii_j (t) = \Delii_{j, h} ( t , \eta, \theta)$ for simplicity. We have that the $\Delii_j (t)$ satisfy the system,
\begin{align} \label{eqn:ii-system}
\del_t \Delii_1 (t) &= - V'' ( u_1 ( t , \eta, \theta ) ) \Delii_1 (t) - Q_1 (t) \cdot k_1 (t , \eta+h, \theta) \notag\\
 \del_t \Delii_j (t) &= -  V'' ( u_j ( t , \eta, \theta ) ) \Delii_j (t) +  V'' ( u_{j-1} ( t , \eta, \theta ) ) \Delii_{j-1} (t) \notag\\
 &+ Q_j (t) \cdot k_j (t, \eta+h , \theta  ) - Q_{j-1} (t) \cdot k_{j-1}(t, \eta+h, \theta ) 
\end{align}
where,
\beq
Q_j (t) := \Deli_{j, h} ( t , \theta, \eta) \int_0^t V''' ( \tau u_j (t, \eta+h, \theta  ) + (1 - \tau ) u_j ( t , \eta , \theta ) ) \d \tau.
\eeq
Let now $f_j$ solve
\begin{align}
\del_t f_1 &= - V'' (u_1 (t, \eta, \theta ) ) f_1 \notag\\
\del_t f_j &= - V'' (u_j (t, \eta, \theta) ) f_j + V'' (u_{j-1} ( t , \eta , \theta ) ) f_{j-1} 
\end{align}
with initial data $f_j (0) = \Delii_j (0)$. By Lemma \ref{lem:ode-2} we have,
\beq
| f_n (t) | \leq 2 \sum_{j=1}^n | \Delii_j (0) |.
\eeq
Let $w_j (t) = \Delii_j (t) - f_j (t)$ so that it satisfies the system of equations \eqref{eqn:ii-system} with initial data $w_j (0) = 0$.  By the assumption that $0 \leq -V''' (x)  \leq C( V''(x)+1)$ we see that,
\beq \label{eqn:2-ii-proof-1}
0 \leq  Q_j (t) \leq -C \Deli_{i, h} ( t , \eta , \theta) (G_j (t)+1)
\eeq
 where $G_j$ are from Proposition \ref{prop:1d-finite-dif-i}.  Applying now Lemma \ref{lem:ode-4}, Proposition \ref{prop:1d-finite-dif-i} to estimate the time integrals of the $G_j$ as well as the estimates for $k_j$ of Corollary \ref{cor:1st-der-2}, we find
\beq
|w_n (t) | \leq Cn (1+t) \left( \sum_{j=1}^n | \Deli_{j,h} ( 0, \eta, \theta) | \right) \left( \sum_{j=1}^n -k_j (0, \eta+h, \theta ) \right) .
\eeq
The claim follows. \qed

\bep
Consider the difference quotients,
\beq
\Delid_{j,h} (t, \eta , \theta) := \frac{ k_j (t , \eta , \theta+h ) - k_j (t , \eta , \theta ) }{h}
\eeq
We have,
\beq
| \Delid_{n, h} (t , \eta , \theta ) | \leq C n t(1+t) \left( \sum_{j=1}^n - k_j (0 , \eta , \theta ) \right).
\eeq
\eep
\proof Denote $\Delid_j (t) = \Delid_{j,h} ( t , \eta , \theta)$ for simplicity. We have that they satisfy,
\begin{align}
\del_t \Delid_j (t) &= - V'' ( u_1 ( t , \eta, \theta ) ) \Delid_1 (t) - Q_1 (t) \cdot k_1 (t , \eta, \theta+h) \notag\\
 \del_t \Delid_j (t) &= -  V'' ( u_j ( t , \eta, \theta ) ) \Delid_j (t) +  V'' ( u_{j-1} ( t , \eta, \theta ) ) \Delid_{j-1} (t) \notag\\
 &+ Q_j (t) \cdot k_j (t, \eta , \theta +h ) - Q_{j-1} (t) \cdot k_{j-1}(t, \eta, \theta+h ) 
\end{align}
where,
\beq
Q_j (t) := \Deld_{j, h} ( t , \theta, \eta) \int_0^t V''' ( \tau u_j (t, \eta, \theta+h  ) + (1 - \tau ) u_j ( t , \eta , \theta ) ) \d \tau ,
\eeq
with initial data $0$.  We now proceed in a similar fashion to the previous two propositions by applying Lemma \ref{lem:ode-3}. The $Q_j$ can be controlled by the $F_j$ of Proposition \ref{prop:1d-finite-dif-d} using \eqref{eqn:2-dd-proof-1}. The time integrals are then estimated using Proposition \ref{prop:1d-finite-dif-d} as well as Corollary \ref{cor:1st-der-2} to control the $k_j$'s. \qed
\bep
Consider the difference quotients,
\beq
\Deldi_{j,h} (t, \eta , \theta) := \frac{ h_j (t , \eta+h , \theta ) - h_j (t , \eta , \theta ) }{h}
\eeq
Then,
\beq
| \Deldi_{n,h} (t, \eta , \theta)  | \leq C n  t ( 1+ t) \left( \sum_{j=1}^n - \Deli (0 , \theta , \eta ) \right)
\eeq
\eep
\proof Denote $\Deldi_{j} (t) = \Deldi_{j, h} (t, \eta , \theta)$ for simplicity. They satisfy,
\begin{align}
\del_t \Deldi_j (t) &= - V'' ( u_1 ( t , \eta, \theta ) ) \Deldi_1 (t) - Q_1 (t) \cdot h_1 (t , \eta+h, \theta) \notag\\
 \del_t \Deldi_j (t) &= -  V'' ( u_j ( t , \eta, \theta ) ) \Deldi_j (t) +  V'' ( u_{j-1} ( t , \eta, \theta ) ) \Deldi_{j-1} (t) \notag\\
 &+ Q_j (t) \cdot h_j (t, \eta +h, \theta  ) - Q_{j-1} (t) \cdot h_{j-1}(t, \eta+h, \theta ) 
\end{align}
with $0$ initial condition where
\beq
Q_j (t) := \Deli_{j, h} ( t , \theta, \eta) \int_0^t V''' ( \tau u_j (t, \eta+h, \theta  ) + (1 - \tau ) u_j ( t , \eta , \theta ) ) \d \tau .
\eeq
We now proceed in a similar fashion to the previous two propositions by applying Lemma \ref{lem:ode-3}. The $Q_j$ can be controlled by the $G_j$ of Proposition \ref{prop:1d-finite-dif-i} using \eqref{eqn:2-ii-proof-1}.  The time integrals are then estimated using Proposition \ref{prop:1d-finite-dif-i} as well as Corollary \ref{cor:1st-der-1} to control teh $h_j$'s. \qed

\subsubsection{Proof of Proposition \ref{prop:2nd-der-system}} \label{a:2-der-summary}

In this section we summarize the proof Proposition \ref{prop:2nd-der-system}. It is straightforward given the four propositions stated and proven in Appendix \ref{a:2-der} for each of the four different sets of finite difference quotients. First, note that Section \ref{a:1-der} establishes that in particular the $u_j (t, \theta, \eta)$ are all jointly continuous in $(t, \eta, \theta)$. Since $V$ is smooth it follows from representing the solutions $h_j (t, \eta, \theta)$ and $k_j (t, \eta, \theta)$ as some combination of iterated integrals (i.e., repeatedly iterating \eqref{eqn:1d-ode}) that they are continuous in $(t, \eta, \theta)$. Then, by representing any of the finite difference quotients $\Deldd_{j,h} (t), \Deldi_{j, h} (t)$, etc., as iterated integrals in a similar fashion, one sees that as $h \to 0$ these converge uniformly to the solutions of the systems of ODEs described in Proposition \ref{prop:2nd-der-system}. We also conclude the continuity of the second derivatives in the parameters. \qed

\subsubsection{Three-parameter height function} \label{a:three}

In Section \ref{sec:three-param} we introduced a three parameter system via \eqref{eqn:3-sys} and associated height function in \eqref{eqn:3-height}. In this section we discuss well-posedness and differentiability of the system. Much of what is needed follows either directly or with similar proofs in two parameter case.

Well-posedness of the system \eqref{eqn:3-sys} can either be proven via the same method as the two-parameter system, relying on Lemma \ref{lem:simple-wp}, or can be constructed via concatenating the solution maps at the time $t = Y N^{2/3}$ that were constructed in the two-parameter case in Appendix \ref{a:diff-1} for each of the $\theta_1$ and $\theta_2$. Due to uniqueness, this results in the same solution.

In order to prove that each of the $\tilu_j (t, \eta, \theta_1, \theta_2)$ are differentiable in the latter three parameters one considers the finite difference quotients as in Section \ref{a:1-der}.  Proceeding in the exact same fashion as in Proposition \ref{prop:1d-finite-dif-d} one finds for $|h| < \theta_1$, 
\beq
0 \leq - \frac{ \tilu_j ( t , \eta , \theta_1 + h , \theta_2 ) - \tilu_j ( t , \eta , \theta_1 + h , \theta_2 )}{ h} \leq t , 
\eeq
and a similar estimate for the difference quotient in $\theta_2$ (in fact, a slightly better estimate is true but not required). For the finite difference quotient in $\eta$ one finds the same estimate as in Proposition \ref{prop:1d-finite-dif-i}. From these estimates, one sees that the $\tilu_j (t, \eta, \theta_1, \theta_2)$ are continuous functions of all four parameters. Then, proceeding as in the proof of Corollaries \ref{cor:1st-der-1} and \ref{cor:1st-der-2}, one finds that they are differentiable in $(\eta, \theta_1, \theta_2)$. 

At this point, one obtains that the first order derivatives satisfy a system of ODEs of triangular form, similar to the cases of $h_j (t)$ and $k_j (t)$ outlined above. These ODEs have explicit representations of integrals of $V$ and its derivatives applied to the functions $\tilu_j ( t , \eta , \theta_1, \theta_2)$. Since these later functions are continuous, it follows easily that the first order derivatives of the $\tilu_j$ are themselves continuous functions of $(\eta, \theta_1, \theta_2)$. One can then easily repeatedly differentiate these integral expressions repeatedly to find that the $\tilu_j (t)$ are $C^k$ in the parameters $(\eta, \theta_1, \theta_2)$. This is sufficient for the purposes of Section \ref{sec:lower}. 

\section{Estimates of moments of initial data derivatives} \label{a:1b}

We require some estimates on the moments of derivatives of the $u_j (0 , \eta, \theta)$ introduced in Definition \ref{def:coupling} with respect to to $\eta$. They are  consequences of some direct calculations which we present in this section.

\bel \label{lem:1b-1}
Let $V$ be a potential of O'Connell-Yor type and let $\eta_0 > 0$. There is a $C>0$ so  that for all $\eta_0 \leq 2 \eta \leq 4 \eta_0$ we have,
\beq
| \del_\eta H_\eta (q) | \leq C (1+| H_\eta (q) | )
\eeq
and
\beq
| \del^2_\eta H_\eta (q) | \leq C (1 + |H_\eta (q) |^2 ) ( 1 + V ( H_\eta (q) ) )
\eeq
\eel
\proof By direct calculation,
\beq
\del_\eta H_\eta (q) = \frac{ \Cov (X , \1_{ \{ X \leq u \} } ) }{F'_\eta ( u ) }
\eeq
where $X$ is distributed according to $\nu_\eta$ and $u = H_\eta (q)$. First assume $u \geq 0$. Then, using that $0 \leq V(x) \leq C$ for $x\geq 0$ we have,
\beq
|  \Cov (X , \1_{ \{ X \leq u \} } ) | =|  \Cov (X , \1_{ \{ X \geq u \} } ) | \leq C \int_u^\infty(1 + s ) \e^{ - \eta s} \d s \leq C(1+u) \e^{ -\eta u},
\eeq
and the desired estimate follows for $u \geq 0$. The convexity of $V$ and the growth assumptions at $-\infty$ imply that $\lim_{x \to -\infty} V'(x) = -\infty$. For $x <y$ using $V(y) - V(x) \leq (y-x) V'(y)$ we see that for $s < u \leq 0$ there is a constant $c>0$ such that,
\beq
V(s) + \eta s \geq V(u) + \eta u +c (u-s) - c^{-1}.
\eeq
So for $u \leq 0$ we have,
\beq
|  \Cov (X , \1_{ \{ X \leq u \} } ) | \leq C \int_{-\infty}^u (1+|s| ) \e^{ - (V (s) + \eta s ) } \d s \leq C(1+ |u| ) \e^{ - V(u) + \eta u }.
\eeq
The estimate for $\del_\eta H_\eta$ follows. By direct calculation,
\begin{align}
\del_\eta^2 H_\eta (q) = - \frac{  F''_\eta (u) ( \del_\eta H_\eta (q) )^2}{ F'_\eta (u) }- 2 \frac{ \del_\eta F' ( u )  ( \del_\eta H_\eta (q) ) }{ F'_\eta (u) } - \frac{ \del_\eta^2 F_\eta (u) }{ F'_\eta (u) }.
\end{align}
By the previous result we have,
\beq
\left|  \frac{  F''_\eta (u) ( \del_\eta H_\eta (q) )^2}{ F'_\eta (u) } \right| \leq C ( V(u) +1)(1+ |H_\eta (q) | )^2
\eeq
where we used that $|V'(x)| \leq C (V(x) + 1)$ (which follows by integrating the second inequality in \eqref{eqn:deriv-assump} twice). Similarly,
\beq
\left| \frac{ \del_\eta F' ( u )  ( \del_\eta H_\eta (q) ) }{ F'_\eta (u) } \right| \leq C (1 +| H_\eta (q) |).
\eeq
By direct calculation,
\begin{align}
\del_\eta^2 F_\eta (u) &= \Cov ( X \1_{ \{ X \leq u \} } , X ) - \ee[X ] \Cov ( \1_{ \{ X \leq u \} } , X ) - \pp[ X \leq u ] \Var (X) \notag\\
&= - \Cov ( X \1_{ \{ X \geq u \} } , X ) + \ee[X ] \Cov ( \1_{ \{ X \geq u \} } , X ) + \pp[ X \geq u ] \Var (X)
\end{align}
For $u \geq 0$, using the second line we easily see in a similar manner as above that,
\beq
 |\Cov ( X \1_{ \{ X \geq u \} } , X )|   + | \pp[ X \geq u ] \Var (X)| \leq C (1 + | H_\eta (q) | )^2 \e^{ - \eta u}.
\eeq
Similarly, using the first line we have for $u \leq 0$ that,
\beq
 | \Cov ( X \1_{ \{ X \leq u \} } , X ) |  +  | \pp[ X \leq u ] \Var (X)  | \leq C (1 + | H_\eta (q) | )^2 \e^{ - V(u) + \eta u }
\eeq
We have already estimated the $ \ee[X ] \Cov ( \1_{ \{ X \leq u \} } , X ) $ term previously. This completes the proof. \qed

\bep \label{prop:1b}
Let $\eta_2 > \eta_1 >0$. Let $q$ be uniform on $(0, 1)$. Define $I:=[ \eta_1 ,  \eta_2]$. For any $ p  \geq 1$  and $k=0, 1, 2$ we have,
\beq \label{eqn:1b-1}
\ee[ \sup_{ \eta \in I } | \del^k_\eta H_\eta (q) |^p ] \leq C_p
\eeq
Define now the difference quotients,
\beq
\Deli_h ( \eta ) := \frac{ H_{\eta + h } ( q ) - H_\eta (q) }{h}
\eeq
and
\beq
\Delii_h (\eta) :=  \frac{ \del_\eta H_{\eta + h } ( q ) - \del_\eta H_\eta (q) }{h} .
\eeq
Then, 
\beq
\ee[ \sup_{ \eta \in I , |h| < \eta_1 /2} | \Deli_h ( \eta ) |^p ] \leq C
\eeq
and
\beq
\ee[ \sup_{ \eta \in I , |h| < \eta_1/2} | \Delii_h ( \eta ) |^p ] \leq C
\eeq
\eep
\proof Since $\eta \mapsto H_\eta$ and $V(x)$ are monotonic functions we see first that,
\beq
\sup_{ \eta \in I } |  H_\eta (q) | +\sup_{ \eta \in I } | V( H_\eta (q)) |  \leq |  H_{\eta_1} (q) | + | V( H_{\eta_1} (q)) |  + |  H_{\eta_2} (q) | + | V( H_{\eta_2} (q)) | 
\eeq
The quantities on the RHS have finite moments of all orders as $H_\eta (q)$ is distributed as $\nu_\eta$. We conclude the estimates \eqref{eqn:1b-1} from Lemma \ref{lem:1b-1}. The estimates for the difference quotients follow from estimating them by the supremum of $| \del_\eta H_\eta| $ and $| \del_\eta^2 H_\eta| $ over the intervals $[\eta_1/2 , 2 \eta_2 ]$. \qed

\bep \label{prop:dif-under}
For $\del = \del_\eta , \del_\theta , \del_\eta^2, \del_\theta^2 ,\del_\eta \del_\theta$ we have,
\beq
\del \ee[ W_{n, t} ( \eta , \theta ) ] = \ee[ \del W_{n, t} (  \eta , \theta ) ]
\eeq
\eep
\proof We claim that this follows from the dominated convergence theorem. First, note that the difference quotients all converge pointwise to the appropriate derivatives of $W_{N, t} ( \eta , \theta)$ by the differentiability established in Appendix \ref{a:diff}. All that is required is an estimate for the supremum of the difference quotients over $h$.  For these, we see that any difference quotient not involving the initial data parameter $\eta$ is bounded by a constant (that may depend on $t$ and $N$ of course) independent of $h$, by the estimates proven in Appendix \ref{a:diff}, i.e., Propositions \ref{prop:1d-finite-dif-d} and \ref{prop:2-dd}. 

For any difference quotient involving the initial data parameter $\eta$ we see by the remaining propositions of Appendix \ref{a:diff} that they are all bounded in terms of the difference quotients or deriviatives of the initial data. By Proposition \ref{prop:1b} the supremum over $h$ and $\eta$ of such quantities have finite moments of all orders, and so we can apply dominated convergence to pass the limit $h \to 0$ inside the integrand, yielding the claim. \qed

\bep \label{prop:Wnt-bad-bd}
We have,
\beq
\ee[ |W_{N, t} ( \eta , \theta ) |^2] \leq CN^2(1+t)^2
\eeq
\eep
\proof We have,
\beq
\ee[ |u_j ( t , \eta, \theta ) |^2] \leq C t^2 + \ee[ |u_j (t , \eta , \eta ) |^2 ] \leq C (t^2+1).
\eeq
This is sufficient to prove the required estimate via Cauchy-Schwarz. \qed

\section{Nonpositivity of $\psiV_2 (\theta)$} \label{a:psiV2}

\bel \label{lem:psiV2}
Let $V$ be a potential of O'Connell-Yor type. Then for all $\theta >0$ we have,
\beq
\psiV_2 ( \theta ) \leq 0.
\eeq
\eel
\proof Consider the solution $v(t, \theta)$ of
\beq \label{eqn:erg-2}
\d v( t) = - V' ( v ( t , \theta ) ) \d t - \theta \d t + \d B_0 (t) + \d B_1 (t) ,
\eeq
with initial data $v(0 , \theta ) = 0$. It follows from the calculations of Appendix \ref{a:2-der} that for all $t$, the function $ \theta \to v( t , \theta)$ is convex in $\theta$. Therefore, for all $h >0$ sufficiently small we have,
\beq
v( t , \theta + h ) + v ( t , \theta - h) - 2 v ( t , \theta ) \geq 0.
\eeq
The claim will then follow by proving that for all $\theta >0$,
\beq \label{eqn:erg-1}
\lim_{t \to \infty} \ee[ v (t , \theta ) ] = \int x \d \nu_\theta (x) ,
\eeq
as the previous two results imply that $2 \psiV_0 ( \theta) - \psiV_0 ( \theta - h) - \psiV_0 ( \theta + h) \geq 0$ for all $\theta $ and $h$ sufficiently small.  

So we turn to proving the convergence \eqref{eqn:erg-1}. Along the way we will also see that the expectations on the LHS are well-defined. Let now $u(t) = u(t, \theta)$ be the solution to \eqref{eqn:erg-2} but with initial data distributed according to the invariant measure $\nu_\theta$ and let $w(t) := u(t) - v(t)$. Then,
\beq
\del_t w(t) = \left( - \int_0^1 V'' ( s u (t) + (1-s) v(t) ) \d s \right) w(t) =: - F(t) w(t).
\eeq
Since $F(t) \geq 0$ we see that $|w(t) | \leq |w(0)|= |u(0)|$ for all $t$ (showing that the expectations in \eqref{eqn:erg-1} are indeed well-defined)  and moreover that $w(t)$ has the same sign as $w(0)$.  The assumption that $V(x) \geq 0$ and the monotonicity of $V'(x)$ implies that $\lim_{x \to \infty} V'(x) = 0$. Therefore, integrating the first inequality of \eqref{eqn:deriv-assump} gives that
\beq
V''(x) \geq - c_0 V'(x) \geq - \eps_0 V' (x)
\eeq
for all $x \in \rr$ and all $c_0 \geq \eps_0 > 0$. It follows that,
\beq
\int_0^1 V'' ( s u (t) + (1-s) v(t) ) \d s \geq \eps_0 \left( - \int_0^1 V' ( s u (t) + (1-s) v(t) ) \d s \right)
\eeq
Since $w(t)$ has the same sign as $w(0)$ it follows that either $v(t) > u(t)$ for all $t$ or $v(t) < u(t)$ for all $t$.  In the first case, using that $V'(x)$ is monotonic, we have,
\begin{align}
\int_0^t - \int_0^1 V' ( r u (s) + (1-r) v(s) ) \d r \d s &\geq \int_0^t - V' ( v(s) ) \d s \notag\\
&= \theta t + v(t) - v(0) -B_0 (t) - B_1 (t) \notag\\
&\geq \theta t - |u(t)| - |u(0)| - |B_0 (t)| - |B_1 (t) |.
\end{align}
In the case $v(t) < u(t)$ we arrive at a similar estimate using the same method.  Let,
\beq
\F_t := \{ |B_0 (t) | + |B_1 (t) | \leq (1 + t^{3/4} ) \}
\eeq
so that $\pp[ \F_t^c] \leq C \e^{- t^{1/10}}$.

It follows that,
\begin{align}
\ee[ |w(t) | ] &\leq \ee[ \exp \left(  - \int_0^t F(s) \d s \right) |w(0) | \1_{\F_t}  ] + \ee[ |w(0)| \1_{\F_t^c} ] \notag\\
&\leq  C \e^{ - \eps_0 \theta t + t^{3/4} } \ee[ \e^{ \eps_0 ( |u(t) | + |u(0)| ) | } |u(0)| ] + C\e^{-t^{1/10}/2} \ee[ |u(0)|^2]^{1/2} \notag\\
&\leq C\e^{ -c t^{1/10} }
\end{align}
as long as $\eps_0 < \theta/100$ so that $\ee[ \e^{ 10 \eps_0 | u(0) | } ] < \infty$. This completes the proof. \qed


\end{document}